\documentclass[11pt,a4paper]{article}

\usepackage{graphicx}
\usepackage{amsmath}
\usepackage{amsfonts}
\usepackage{amssymb}
\usepackage{enumerate}
\usepackage{lscape}
\usepackage{longtable}
\usepackage{rotating}
\usepackage{color}
\usepackage{url}
\usepackage{rotating}
\usepackage{algpseudocode}
\usepackage[ruled]{algorithm}
\usepackage{pgf,tikz}
\usepackage{subfig}
\usepackage[titletoc]{appendix}

\usepackage{mathrsfs}
\usepackage{pgfplots}
\usetikzlibrary{arrows}
\usepackage{color}
\usepackage{booktabs}
\usepackage{multirow}
\usepackage{mathtools}
\usepackage{makecell}

\newtheorem{theorem}{Theorem}[section]

\newtheorem{assumption}{Assumption}[section]

\newenvironment{proof}[1][Proof]{\noindent \textbf{#1.} }{\hfill$\Box$\par\medskip}

\DeclareMathOperator*{\argmax}{argmax}

\DeclareMathOperator*{\argmin}{argmin}

\newcommand{\beqn}[1]{\begin{equation}\label{#1}}
\newcommand{\eeqn}{\end{equation}}

\definecolor{aau2}{rgb}{0.0, 0.5, 0.69}
\definecolor{aau3}{rgb}{0.0, 0.53, 0.74}
\definecolor{aau4}{rgb}{0.0, 0.48, 0.65}
\definecolor{aau5}{rgb}{0.0, 0.45, 0.73}
\definecolor{rsap}{RGB}{130, 36, 51}
\definecolor{gsap}{RGB}{112, 164, 137}

\definecolor{tud}{rgb}{0.43,0.73,0.11}
\definecolor{verde}{rgb}{0.33,0.53,0.11}

\definecolor{ttffqq}{rgb}{0.0, 0.48, 0.65} 
\definecolor{ffqqqq}{rgb}{0.0, 0.5, 0.69} 

\usetikzlibrary{arrows}

\tikzstyle{decision} = [diamond, draw, fill=blue!20,
    text width=4.5em, text badly centered, node distance=3cm, inner sep=0pt]
\tikzstyle{block} = [rectangle, draw, fill=blue!20,
     text centered, rounded corners, minimum height=4em]
\tikzstyle{line} = [draw, -latex']
\tikzstyle{cloud} = [draw, ellipse,fill=red!20, node distance=3cm,
    minimum height=2em]
\tikzstyle{cloud2} = [draw, ellipse,fill=green!20, node distance=3cm,
    minimum height=2em]
\begin{document}

\title{The stochastic multi-gradient algorithm for multi-objective optimization and its application to supervised machine learning}

\date{\today}

\author{
S. Liu\thanks{Department of Industrial and Systems Engineering,
Lehigh University,
200 West Packer Avenue, Bethlehem, PA 18015-1582, USA
({\tt sul217@lehigh.edu}).}\and
L. N. Vicente\thanks{Department of Industrial and Systems Engineering,
Lehigh University,
200 West Packer Avenue, Bethlehem, PA 18015-1582, USA and
Centre for Mathematics of the University of Coimbra (CMUC)
({\tt lnv@lehigh.edu}). Support for
this author was partially provided by FCT/Portugal under grants
UID/MAT/00324/2019 and P2020 SAICTPAC/0011/2015.}
}

\maketitle

{\small
\begin{abstract}
Optimization of conflicting functions is of paramount importance in decision making, and real world applications frequently involve data that is uncertain or unknown, resulting in multi-objective optimization (MOO) problems of stochastic type. We study the stochastic multi-gradient (SMG) method, seen as an extension of the classical stochastic gradient method for single-objective optimization.

At each iteration of the SMG method, a stochastic multi-gradient direction is calculated by solving a quadratic subproblem, and it is shown that this direction is biased even when all individual gradient estimators are unbiased. We establish rates to compute a point in the Pareto front, of order similar to what is known for stochastic gradient in both convex and strongly convex cases. The analysis handles the bias in the multi-gradient and the unknown a priori weights of the limiting Pareto point.

The SMG method is framed into a Pareto-front type algorithm for calculating an approximation of the entire Pareto front. The Pareto-front SMG algorithm is capable of robustly determining Pareto fronts for a number of synthetic test problems. One can apply it to any stochastic MOO problem arising from supervised machine learning, and we report results for logistic binary classification where multiple objectives correspond to distinct-sources data groups.
\end{abstract}

\bigskip

\begin{center}
\end{center}
}

\section{Introduction}
\label{introduction}

In multi-objective optimization (MOO) one attempts to simultaneously optimize several, potentially conflicting functions. MOO has wide applications in all industry sectors where decision making is involved due to the natural appearance of conflicting objectives or criteria. Applications span across applied engineering, operations management, finance, economics and social sciences, agriculture, green logistics, and health systems. When the individual objectives are conflicting, no single solution exists that optimizes all of them simultaneously. In such cases, the goal of MOO is then to find Pareto optimal solutions (also known as efficient points), roughly speaking points for which no other combination of variables leads to a simultaneous improvement in all objectives. The determination of the set of Pareto optimal solutions helps decision makers to define the best trade-offs among the several competing criteria.

We start by introducing an MOO problem~\cite{MEhrgott_2005, KMiettinen_2012} consisting of the simultaneous minimization of~$m$ individual functions
\begin{equation}
\begin{array}{rl}
    \min & H(x) = (h_1(x), \ldots, h_m(x))^\top  \\ [0.5ex]
    \text{s.t.} & x \in \mathcal{X},
\end{array}
\label{MOOP_formulation}
\end{equation}
where $h_i: \mathbb{R}^n \to \mathbb{R}$ are real valued functions and
$\mathcal{X} \subseteq \mathbb{R}^n$ represents a feasible region.
We say that the MOO problem is smooth if all objective functions~$h_i$ are continuously differentiable. Assuming that no point may exist that simultaneously minimizes all objectives, the notion of Pareto dominance is introduced to compare any two given feasible points $x,y \in \mathcal{X}$.
One says that~$x$ dominates~$y$ if $H(x) < H(y)$ componentwise. A point~$x \in \mathcal{X}$ is a Pareto minimizer if it is not dominated by any other point in~$\mathcal{X}$. The set of all Pareto minimizers includes the possible multiple minimizers of each individual function. If we want to exclude such points, one can consider the set of strict Pareto minimizers~$\mathcal{P}$, by rather considering a weaker form of dominance (meaning $x$ weakly dominates~$y$ if $H(x) \leq H(y)$ componentwise and $H(x) \neq H(y)$)\footnote{Although not necessary for this paper, note that both dominance conditions above stated induce a strict partial ordering of the points in~$\mathbb{R}^m$. Also, subjacent to such orderings is the cone corresponding to the nonnegative orthant $K = \{ v \in \mathbb{R}^m: v_i \geq 0, i=1,\ldots,m \}$. In fact, $x$ dominates $y$ if and only if $H(y)-H(x) \in \mbox{int}(K)$, and
$x$ weakly dominates $y$ if and only if $H(y)-H(x) \in K\setminus \{0\}$. Broadly speaking any pointed convex cone~$K$ will induce in these two ways a strict partial ordering.}. In this paper we can broadly speak of Pareto minimizers as the first-order optimality condition considered will be necessary for both Pareto optimal sets.
An important notion in MOO is the Pareto front $H(\mathcal{P})$, formed by mapping all elements of $\mathcal{P}$ into the decision space~$\mathbb{R}^m$,
$H(\mathcal{P})= \{H(x): x \in \mathcal{P}\}$.

\subsection{Deterministic multi-objective optimization}
\label{deter_MOOP_review}

If one calculates the Pareto front, or a significant portion of it, using an \textit{a posteriori} methodology, then decision-making preferences can then be expressed upon the determined Pareto information. This approach contrasts with methods that require an \textit{a priori} input from the decision maker in the decision space, such as a utility function or a ranking of importance of the objectives~\cite{MEhrgott_2005,KMiettinen_2012}.

A main class of MOO methods apply scalarization, reducing the MOO problem to a single objective one, whose solution is a Pareto minimizer.
These methods require an {\it a priori} selection of the parameters to produce such an effect. The weighted-sum method is simply to assign each objective function $h_i(x)$ a nonnegative weight~$a_i$ and minimize the single objective $S(x,a)=\sum_{i = 1}^m a_i h_i(x)$ subject to the problem constraints (see, for instance,~\cite{SGass_TSaaty_1955}).
If all objective functions are convex, by varying the weights in a simplex set one is guaranteed to determine the entire Pareto front.
The $\epsilon$--constraint method~\cite{YVHaimes_1971} consists of minimizing one objective, say $h_i(x)$, subject to additional constraints that $h_j(x) \leq  \epsilon_j$ for all $j \neq i$, where $\epsilon_j \geq \min_{x \in \mathcal{X}} h_j(x)$ is an upper bound $h_j$ is allowed to take.
Although scalarization methods are conceptually simple, they exhibit some drawbacks: 1) Weights are difficult to preselect, especially when objectives have different magnitudes. Sometimes, the choice of parameters can be problematic, e.g., producing infeasibility in the $\epsilon$--constraint method;  2) In the weighted-sum method, it is frequently observed (even for convex problems) that an evenly distributed set of weights in a simplex fails to produce an even distribution of Pareto minimizers in the front. 3) It might be impossible to find the entire Pareto front if some of the objectives are nonconvex, as it is the case for the weighted-sum method.
There are scalarization methods which have an {\it a posteriori} flavor like
the so-called normal boundary intersection method~\cite{IDas_JEDennis_1998}, and which are able to produce a more evenly distributed set of points on the Pareto front given an evenly distributed set of weights (however solutions of the method subproblems may be dominated points in the nonconvex case~\cite{EHFukuda_LMGDrummond_2014}).

Nonscalarizing {\it a posteriori} methods attempt to optimize the individual objectives simultaneously in some sense. The methodologies typically consist of iteratively updating a list of nondominated points, with the goal of approximating the Pareto front.
To update such iterate lists, some of these {\it a posteriori} methods borrow ideas from population-based heuristic optimization, including Simulated Annealing, Evolutionary Optimization, and Particle Swarm Optimization. NSGA-II \cite{KDeb_etal_2002} and AMOSA \cite{SBandyopadhyay_etal_2008} are two well-studied population-based heuristic algorithms designed for MOO. However, no theoretical convergence properties can be derived under reasonable assumptions for these methods, and they are slow in practice due to the lack of first-order principles.
Other {\it a posteriori} methods update the iterate lists by applying steps of rigorous MOO algorithms designed for the computation of a single point in the Pareto front. Such rigorous MOO algorithms have resulted from generalizing classical algorithms of single-objective optimization to MOO.

As mentioned above, a number of rigorous algorithms have been developed for MOO by extending single-objective optimization counterparts. A common feature of these MOO methods is the attempt to move along a direction that simultaneously decreases all objective functions. In most instances it is possible to prove convergence to a first-order stationary Pareto point. Gradient descent is a first example of such a single-objective optimization technique that led to the multi-gradient (or multiple gradient) method for MOO~\cite{JFliege_BFSvaiter_2000} (see also \cite{LGDrummond_BFSvaiter_2005,LGDrummond_ANIusem_2004,JADesideri_2012,JADesideri_2014}).
As analyzed in \cite{JFliege_AIFVaz_LNVicente_2018}, it turns out that the multi-gradient method proposed by \cite{JFliege_BFSvaiter_2000} shares the same convergence rates as in the single objective case, for the various cases of nonconvex, convex, and strongly convex assumptions.
Other first-order derivative-based methods that were extended to MOO include proximal methods~\cite{HBonnel_ANIusem_BFSvaiter_2005}, nonlinear conjugate gradient methods~\cite{LRLucambioPerez_LFPrudente_2018}, and trust-region methods~\cite{SQu_MGoh_BLiang_2013,KDVillacorta_PROliveira_ASoubeyran_2014}.
Newton's method for multi-objective optimization, further using second-order information, was first presented in \cite{JFliege_LGDrummond_BFSvaiter_2009} and later studied in \cite{LGDrummond_FMPRaupp_BFSvaiter_2014}. For a complete survey on multiple gradient-based methods see~\cite{EHFukuda_LMGDrummond_2014}.
Even when derivatives of the objective functions are not available for use, rigorous techniques were extended along the same lines from one to several objectives, an example being the
the so-called direct multi-search algorithm~\cite{ALCustodio_etal_2011}.

\subsection{Stochastic multi-objective optimization}
\label{stocha_MOOP_review}

\subsubsection{Single objective}
\label{subsec:SO}

Many practical optimization models involve data parameters that are unknown or uncertain, examples being demand or return. In some cases the parameters are confined to sets of uncertainty, leading to robust optimization, where one tries to find a solution optimized against a worst-case scenario.
In stochastic optimization/programming, data parameters are considered random variables, and frequently some estimation can be made about their probability distributions. Let us consider the unconstrained optimization of a single function $f(x,w)$ that depends on the decision variable $x$ and on the random variable/parameter~$w$.
The goal of stochastic optimization is to seek a solution that optimizes the expectation of~$f$ taken with respect to the random variable
\begin{equation}
\label{ml_expected_risk}
\min \; f(x) \;=\; \mathbb{E}[f(x, w)],
\end{equation}
where $w$ is a random variable defined in a probability space 
(with probability measure independent from $x$), for which we assume that i.i.d. samples can be observed or generated. 
An example of interest to us is classification in supervised machine learning, where one wants to build a predictor (defined by $x$) that maps features into labels (the features and labels can be seen as realizations of~$w$) by minimizing some form of misclassification. The objective function~$f(x)$ in~(\ref{ml_expected_risk}) is then called the expected risk (of misclassification), for which there is no explicit form since pairs of features and labels are drawn according to a unknown distribution.

There are two widely-used approaches for solving problem~\eqref{ml_expected_risk}, the sample average approximation (SAA) method and the stochastic approximation (SA) method. Given $N$ i.i.d. samples $\{w^j\}_{j = 1}^N$, one optimizes in
SAA~(see~\cite{AJKleywegt_etal_2002,AShapiro_2003}) an empirical approximation of the expected risk
\begin{equation}
\label{ml_empirical_risk}
\min \; f^N(x) \;=\; \frac{1}{N} \sum_{j = 1}^N f(x, w^j).
\end{equation}
The SA method becomes an attractive approach in practice
when the explicit form of the gradient $\nabla f(x)$ for~\eqref{ml_expected_risk} is not accessible or the gradient $\nabla  f^N(x)$ for~\eqref{ml_empirical_risk} is too expensive to compute when $N$ is large. The earliest prototypical SA algorithm, also known as stochastic gradient (SG) algorithm, dates back to the paper~\cite{HRobbins_SMonro_1951}; and the classical convergence analysis goes back to the works~\cite{KLChung_1954,JSacks_1958}. In the context of solving~\eqref{ml_expected_risk}, the SG algorithm is defined by $x_{k+1} = x_k - \alpha_k \nabla f(x_k, w_k)$, where $w_k$ is a copy of $w$, $\nabla f(x_k, w_k)$ is a stochastic gradient generated according to $w_k$, and $\alpha_k$ is a positive stepsize. When solving problem~\eqref{ml_empirical_risk}, a realization of $w_k$ may just be a random sample uniformly taken from $\{w^1, \ldots, w^N\}$. Computing the stochastic gradient $-\nabla f(x_k, w_k)$ based on a single sample makes each iterate of the SG algorithm very cheap. However, note that only the expectation of $-\nabla f(x_k, w_k)$ is descent for~$f$ at $x_k$, and therefore the performance of the SG algorithm is quite sensitive to the variance of the stochastic gradient. A well-known idea to improve its performance is the use of a batch gradient at each iterate, namely updating each iterate by $x_{k+1} = x_k - \frac{\alpha_k}{|S_k|}\sum_{j \in S_k}\nabla f(x_k, w^j)$, where $S_k$ is a minibatch sample from $\{1, \ldots, N\}$ of size $|S_k|$.  More advanced variance reduction techniques can be found in~\cite{ADefazio_FBach_SLacoste-Julien_2014,RJohnson_TZhang_2013,BTPolyak_ABJuditsky_1992,SShalev-Shwartz_etal_2011} (see the review~\cite{LBottou_FECurtis_JNocedal_2018}).

\subsubsection{Multiple objectives}
Stochastic multi-objective optimization studies scenarios involving both uncertainty and multiple objectives, and founds applications in fields such as finance, energy, transportation, logistics, and supply chain management~\cite{WJGutjahr_APichler_2016}. Generally, when the individual objectives have the form in \eqref{ml_expected_risk}, we face a stochastic multi-objective optimization (SMOO) problem:
\begin{equation}
\begin{array}{rl}
     \min & F(x) = (f_1(x), \ldots, f_m(x))^\top =  (\mathbb{E}[f_1(x,w)], \ldots, \mathbb{E}[f_m(x, w)])^\top \\ [0.5ex]
     \mbox{s.t.} & x \in \mathcal{X},
\end{array}
\label{stochastic_moo}
\end{equation}
where $f_i(x)$ denotes now the $i$-th objective function value and $f_i(x, w)$ is the $i$-th stochastic function of the decision variable $x$ and the random variable $w$. According to the above formulation, multiple objectives can involve common random variables. A typical example is stochastic vehicle routing problem~\cite{MGendreau_OJabali_WRei_2014, JOyola_HArntzen_DLWoodruff_2018} for which one may want to minimize total transportation costs and maximize customer satisfaction simultaneously given stochastic information about demand, weather, and traffic condition. It is also possible that different objectives depend on independent random information. In either case, we denote all the random information/scenarios by the random variable $w$.
%

In the finite sum case~(\ref{ml_empirical_risk}) one has that $\mathbb{E}[f_i(x,w)]$
is equal to or can be approximated by
\begin{equation}
\label{ml_empirical_risk_moo}
f_i^N(x) \;=\; \frac{1}{N} \sum_{j = 1}^N f_i(x, w^j).
\end{equation}
Our work assumes that $\mathcal{X}$ does not involve uncertainty (see the survey~\cite{FBAbdelaziz_2012}) for problems with both stochastic objectives and constraints).

The main approaches for solving the SMOO problems are classified into two categories~\cite{FBAbdelaziz_1992, FBAbdelaziz_2012, RCaballero_2004}: the \textit{multi-objective} methods and the \textit{stochastic} methods. The multi-objective methods first reduce the SMOO problem into a deterministic MOO problem, and then solve it by techniques for deterministic MOO (see Subsection~\ref{deter_MOOP_review}). The stochastic methods first aggregate the SMOO problem into a single objective stochastic problem and then apply single objective stochastic optimization methods (see Subsection~\ref{subsec:SO}). Both approaches have disadvantages~\cite{RCaballero_2004}. Note that the stochastic objective functions $f_i$, $i = 1, \ldots, m$, may be correlated to each other as they possibly involve the same set of random information given by~$w$. Without taking this possibility into consideration, the multi-objective methods might simplify the problem by converting each stochastic objective to a deterministic counterpart independently of each other. As for the stochastic methods, they obviously inherit the drawbacks of {\it a priori} scalarizarion methods for deterministic MOO. We will nevertheless restrict our attention to multi-objective methods by assuming that the random variables in the individual objectives are independent of each other.

\subsection{Contributions of this paper}

This paper contributes to the solution of stochastic multi-objective optimization (SMOO) problems of the form~\eqref{stochastic_moo} by providing an understanding of the behavior and basic properties of the stochastic multi-gradient (SMG) method, also called stochastic multiple
gradient method~\cite{MQuentin_PFabrice_JADesideri_2018}. Stochastic gradient descent is well studied and understood for single-objective optimization. Deterministic multi-gradient descent is also well understood for MOO. However, little is known yet about stochastic multi-gradient descent for stochastic MOO. The authors in~\cite{MQuentin_PFabrice_JADesideri_2018} introduced and analyzed it, but failed to identify a critical property about the stochastic multi-gradient direction, and as a consequence analyzed the method under strong
and unnatural assumptions. Moreover, they did not present its extension from an algorithm that produces a single point in the Pareto front to one that computes an approximation of the entire Pareto front.

The steepest descent direction for deterministic MOO results from the solution of a subproblem where one tries to compute a direction that is the steepest among all functions, subject to some form of Euclidean regularization. The dual of this problem shows that this is the same as calculating the negative of the minimum-norm convex linear combination of all the individual gradients (and the coefficients of such a linear convex combination form a set of simplex weights). From here it becomes then straightforward how to compute a direction to be used in the stochastic multi-gradient method, by simply feeding into such a subproblem unbiased estimations of the individual gradients (the corresponding weights are an approximation or estimation of the true ones). However it turns out the Euclidean regularization or minimum-norm effect introduces a bias in the overall estimation. A practical implementation and a theoretical analysis of the method have necessarily to take into account the biasedness of the stochastic multi-gradient.

In this paper we first study the bias of the stochastic multi-gradient direction and derive a condition for the amount of biasedness that is tolerated to achieve convergence at the appropriate rates. Such a condition will depend on the stepsize but can be enforced by increasing the batch size used to estimate the individual gradients. Another aspect that introduces more complexity in the MOO case is not
knowing the limiting behavior of the approximate weights generated by the algorithm when using sampled gradients, or even of the true weights if the subproblem would be solved using the true gradients. In other words, this amounts to say that we do not know which point in the Pareto front the algorithm is targeting at.
We thus develop a convergence analysis measuring the expected gap between
$S(x_k,\lambda_k)$ and $S(x_*, a_k )$, for various possible selections of $a_k$ as approximations for $\lambda_*$, where $x_k$ is the current iterate, $\lambda_k$ are the true weights, and
$x_*$ is a Pareto optimal solution (with corresponding weights~$\lambda_*$). The choice $a_k = \lambda_*$ requires however a stronger assumption, essentially saying that $\lambda_k$ identifies well the optimal role of $\lambda_*$.
Our convergence analysis shows that the stochastic multi-gradient algorithm exhibits convergence rates similar as in the single stochastic gradient method, i.e., $\mathcal{O}(1/k)$ for strongly convexity and $\mathcal{O}(1/\sqrt{k})$ for convexity.

The practical solution of many MOO problems requires however the calculation of the entire Pareto front. Having such a goal in mind also for the stochastic MOO case, we propose a Pareto-front multi-gradient stochastic (PF-SMG) method that iteratively updates a list of nondominated points by applying a certain number of steps of the stochastic multi-gradient method at each point of the list. Such process generates a number of points which are then added to the list. The main iteration is ended by
removing possible dominated points from the list.
We tested our Pareto-front stochastic multi-gradient method, using synthesis MOO problems~\cite{ALCustodio_etal_2011} to which noise was artificially added, and then measured the quality of the approximated Pareto fronts in terms of the so-called Purity and Spread metrics. The new algorithm shows satisfactory performance when compared with a corresponding deterministic counterpart.

We have applied the Pareto-Front SMG algorithm to stochastic MOO problems arising from supervised machine learning, in the setting of logistic binary classification where multiple objectives correspond to different sources of data within a set.
The determination of the Pareto front can help identifying classifiers that trade-off such sources or contexts, thus improving the {\it fairness} of the classification process.

\subsection{Organization of this paper}

In the context of deterministic MOO, we review in Section~\ref{MGDA_direction} the first-order necessary condition for Pareto optimality and the subproblems for computing the common descent direction, denoted here by multi-gradient direction. Section~\ref{SMG_description} introduces the Stochastic Multi-Gradient (SMG) algorithm in detail, and Section~\ref{biasedness_analysis} reports on the existence of biasedness in the stochastic multi-gradients used in the algorithm. The convergence rates for both convex and strongly convex cases are derived in Section~\ref{convergence_analysis_smg}. The Pareto-Front Stochastic Multi-Gradient algorithm (PF-SMG) is outlined in Section~\ref{section_psmg_pmg}. Our numerical experiments for synthetic and machine learning problems are reported in Section~\ref{numerical_result}, and the paper is concluded in Section~\ref{conclude} with final remarks and prospects of future work.

\section{Pareto stationarity and common descent direction in the deterministic multi-objective case}
\label{MGDA_direction}

The simplest descent method for solving smooth unconstrained MOO problems, i.e. problem~\eqref{MOOP_formulation} with $\mathcal{X} = \mathbb{R}^n$, is the multi-gradient method proposed originally in~\cite{JFliege_BFSvaiter_2000} and further developed in~\cite{JADesideri_2012,LGDrummond_BFSvaiter_2005}. Each iterate takes a step of the form $x_{k+1} = x_k + \alpha_kd_k$, where $\alpha_k$ is a positive stepsize and $d_k$ is a common descent direction at the current iteration $x_k$.

A necessary condition for a point $x_k$ to be a (strict or nonstrict) Pareto minimizer of~\eqref{MOOP_formulation} is that there does not exist any direction that is first-order descent for all the individual objectives, i.e.,
\begin{equation}
\label{pareto_stationarity}
    \text{range}\left (\nabla J_H(x_k)\right) \cap (-\mathbb{R}^m_{++}) \; = \; \emptyset, 
\end{equation}
where $\mathbb{R}^m_{++}$ is the positive orthant cone and $\nabla J_H(x_k)$ denotes the Jacobian matrix of~$H$ at $x_k$.
Condition~\eqref{pareto_stationarity} characterizes first-order Pareto stationary. In fact, at such a nonstationary point $x_k$, there must exist a descent direction $d \in \mathbb{R}^n$ such that $\nabla h_i(x_k)^\top d < 0$, $i=1,\ldots,m$, and one could decrease all functions along~$d$.

When $m = 1$ we simply take $d_k = -\nabla h_1(x_k)$ as the steepest descent or negative gradient direction, and this amounts to minimize $\nabla h_1(x_k)^\top d + (1/2)\|d\|^2$ in $d$. In MOO ($m>1$), the steepest common descent direction~\cite{JFliege_BFSvaiter_2000} is defined by minimizing the amount of first-order Pareto stationarity, also in a regularized Euclidean sense,
\begin{equation}\label{subproblem1}
\begin{split}
(d_k, \beta_k) \;\in\; \displaystyle \argmin_{d \in \mathbb{R}^n, \beta \in \mathbb{R}} &\quad \beta + \frac{1}{2} \| d \|^2  \\
\mbox{s.t.} & \quad \nabla h_i(x_k)^\top d - \beta \leq 0, \; \forall i = 1,\ldots, m.
\end{split}
\end{equation}
If~$x_k$ is first-order Pareto stationary, then $(d_k,\beta_k)=(0,0) \in \mathbb{R}^{n+1}$, and if not, $\nabla h_i(x_k)^\top d_k \leq \beta_k < 0$, for all $i = 1,\ldots, m$ (see~\cite{JFliege_BFSvaiter_2000}). The direction~$d_k$ minimizes
$\max_{1 \le i \le m} \{\nabla h_i(x_k)^\top d\} + (1/2) \|d\|^2$.

It turns out that the dual of~\eqref{subproblem1} is the following subproblem
\begin{equation} \label{subproblem2}
\begin{split}
\lambda_k \;\in\; \displaystyle \argmin_{\lambda \in \mathbb{R}^m} & \quad \left\| \sum^m_{i = 1} \lambda_i \nabla h_i(x_k) \right\|^2  \\
 \mbox{s.t.} & \quad \lambda \in \Delta^m,
\end{split}
\end{equation}
where $\Delta^m = \{\lambda: \sum^m_{i = 1} \lambda_i = 1,\lambda_i \geq 0,  \forall i = 1,...,m\}$ denotes the simplex set.  Subproblem~\eqref{subproblem2} reflects the fact that the common descent direction is pointing opposite to the minimum-norm vector in the convex hull of the gradients $\nabla h_i(x_k), i = 1, \ldots, m$. Hence, the common descent direction, called in this paper a negative multi-gradient, is written as $d_k =  - \sum^m_{i = 1} (\lambda_k)_i \nabla h_i(x_k)$. In the single objective case ($m = 1$), one recovers $d_k = - \nabla h_1(x_k)$.
If~$x_k$ is first-order Pareto stationary, then the convex hull of the individual gradients contains the origin, i.e.,
\begin{equation}
\label{pareto_stationarity2}
\exists \lambda \; \in \; \Delta^m \text{  such that  } \sum_{i = 1}^m \lambda_i \nabla h_i(x_k) \;=\; 0.
\end{equation}
When all the objective functions are convex, we have $x_k \in \mathcal{P}$ if and only if $x_k$ is Pareto first-order stationary~\cite{AMGeoffrion_1968,KMiettinen_2012}.

The multi-gradient algorithm~\cite{JFliege_BFSvaiter_2000} consists of taking $x_{k+1} = x_k + \alpha_k d_k$, where $d_k$ results from the solution of any of the above subproblems and $\alpha_k$ is a positive stepsize.
The norm of $d_k$ is a natural stopping criterion.
Selecting $\alpha_k$ either by backtracking until an appropriate sufficient decrease condition is satisfied or by taking a fixed stepsize inversely proportional to the maximum of the Lipschitz constants of the gradients of the individual gradients leads to the classical sublinear rates of $1/\sqrt{k}$ and $1/k$ in the nonconvex and convex cases,
respectively, and to a linear rate in the strongly convex case~\cite{JFliege_AIFVaz_LNVicente_2018}.

\section{The stochastic multi-gradient method}
\label{SMG_description}

Let us now introduce the stochastic multi-gradient (SMG) algorithm for the solution of the stochastic MOO problem~\eqref{stochastic_moo}.
For this purpose let~$\{w_k\}_{k \in \mathbb{N}}$ be a sequence of copies of the random variable $w$. At each iteration we sample stochastic gradients~$g_i(x_k, w_k)$ as approximations of the true gradients $\nabla f_i(x_k)$, $i=1\ldots,m$. The stochastic multi-gradient is then computed by replacing the true gradients $\nabla f_i(x_k)$ in subproblem~\eqref{subproblem2} by the corresponding stochastic gradients~$g_i(x_k, w_k)$, leading to the following subproblem:
\begin{equation} \label{subproblem3}
\begin{split}
\lambda^g(x_k, w_k) \;\in\; \displaystyle \argmin_{\lambda \in \mathbb{R}^m} & \quad \left\| \sum^m_{i = 1} \lambda_i g_i(x_k, w_k) \right\|^2  \\
\mbox{s.t.} & \quad \lambda \in \Delta^m,
\end{split}
\end{equation}
where the convex combination coefficients $\lambda^g_k = \lambda^g(x_k, w_k)$ depend on~$x_k$ and on the random variable $w_k$. Let us denote the stochastic multi-gradient by
\begin{equation} \label{smg}
g(x_k, w_k) \; = \; \sum^m_{i = 1} \lambda_i^g(x_k, w_k) g_i(x_k, w_k).
\end{equation}

Analogously to the unconstrained deterministic case, each iterative update of the SMG algorithm takes the form~$x_{k+1} = x_k - \alpha_kg(x_k, w_k)$, where $\alpha_k$ is a positive step size and $g(x_k, w_k)$ is the stochastic multi-gradient. More generally, when considering a closed and convex constrained set~$\mathcal{X}$ different from $\mathbb{R}^n$, we need to first orthogonally  project $x_k - \alpha_kg(x_k, w_k)$ onto~$\mathcal{X}$ (such projection is well defined and results from the solution of a convex optimization problem). The SMG algorithm is described as follows.

{\linespread{1}\addtocounter{algorithm}{-1}
\renewcommand{\thealgorithm}{{1}} 
\begin{algorithm}[H] 
\caption{Stochastic Multi-Gradient (SMG) Algorithm} 
\label{alg:SMG} 
\begin{algorithmic}[1] 
\par\vspace*{0.1cm}
\item
Choose an initial point $x_0 \in \mathcal{X}$ and a step size sequence $\{\alpha_k\}_{k \in \mathbb{N}} > 0$.
\item {\bf for} $k=0,1,  \ldots$ {\bf do}
 \item \quad Compute the stochastic gradients $g_i(x_{k}, w_k)$ for the individual functions, $i = 1, \ldots, m$.
  \item \quad Solve problem~\eqref{subproblem3} to obtain the stochastic multi-gradient~(\ref{smg}) with $\lambda^g_k \in \Delta^m$.
\item \quad Update the next iterate $x_{k + 1} = P_{\mathcal{X}}(x_k - \alpha_k g(x_k, w_k))$.
\item {\bf end for}
\par\vspace*{0.1cm}
\end{algorithmic}
\end{algorithm}
}
As in the stochastic gradient method, there is also no good stopping criterion for the SMG algorithm, and one may have just to impose a maximum number of iterations.

\section{Biasedness of the stochastic multi-gradient}
\label{biasedness_analysis}

Figure~\ref{fig:SGM} provides us the intuition for Subproblems~\eqref{subproblem2} and~\eqref{subproblem3} and their solutions when $n=m=2$.
In this section for simplicity we will omit the index~$k$.
Let $g_1^1$ and $g_2^1$ be two unbiased estimates of the true gradient $\nabla f_1(x)$ for the first objective function, and $g_1^2$ and $g_2^2$ be two unbiased estimates of the true gradient $\nabla f_2(x)$ for the second objective function. Then, $g_1$ and~$g_2$, the stochastic multi-gradients from solving~\eqref{subproblem3}, are estimates of the true multi-gradient $g$ obtained from solving~\eqref{subproblem2}.

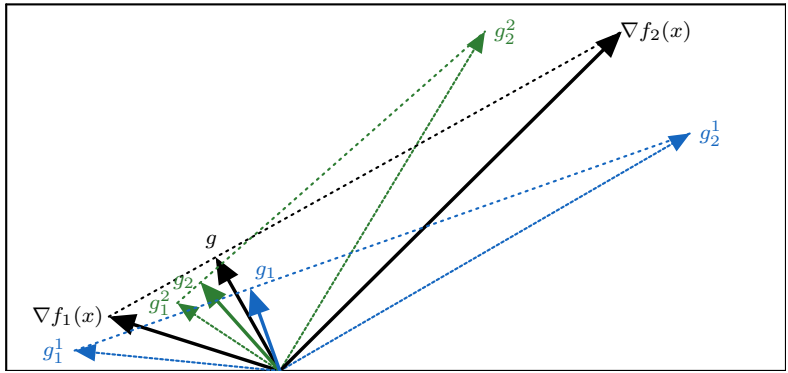
\begin{figure}[H]
\centering
\definecolor{sexdts}{rgb}{0.18,0.49,0.20} 
\definecolor{ffqqqq}{rgb}{0.082,0.40,0.75} 
\definecolor{rvwvcq}{rgb}{0,0,0} 

\begin{tikzpicture}[line cap=round,line join=round, >=triangle 45,x=1cm,y=1cm, scale=0.9]
\clip(-5.5, 0) rectangle (7.4, 5.5);

\draw[thick] (-4, 0) -- (7.4, 0) node[anchor=north west]{};
\draw[thick] (-4, 5.4) -- (7.4, 5.4) node[anchor=south east]{};
\draw[thick] (-4, 0) -- (-4, 5.4) node[anchor=north west]{};
\draw[thick] (7.4, 0) -- (7.4, 5.4) node[anchor=south east]{};

\draw [->,line width=0.8pt,dash pattern=on 1pt off 1pt,color=sexdts] (0,0) -- (3,5);
\draw [->,line width=0.8pt,dash pattern=on 1pt off 1pt,color=ffqqqq] (0,0) -- (6,3.5);
\draw [->,line width=1.3pt,color=rvwvcq] (0,0) -- (-2.5,0.8);
\draw [->,line width=0.8pt,dash pattern=on 1pt off 1pt,color=ffqqqq] (0,0) -- (-3,0.3);
\draw [->,line width=0.8pt,dash pattern=on 1pt off 1pt,color=sexdts] (0,0) -- (-1.5,1);

\draw [->,line width=1.3pt,color=rvwvcq] (0,0) -- (5,5);
\draw [->,line width=1.3pt, color=rvwvcq] (0,0) -- (-0.938, 1.675);
\draw [->,line width=1.3pt, color=sexdts] (0,0) -- (-1.162,1.311);

\draw [line width=0.8pt,dotted, color=rvwvcq] (-2.5,0.8)-- (5,5);
\draw [line width=0.8pt,dotted, color=ffqqqq] (-3,0.3)-- (6,3.5);
\draw [line width=0.8pt,dotted, color=sexdts] (-1.5,1)-- (3,5);
\draw [->,line width=1.3pt,color=ffqqqq] (0,0) -- (-0.431, 1.213);

\begin{scriptsize}
\draw[color=rvwvcq] (-3.1, 0.8) node {$\nabla f_1(x)$};
\draw[color=ffqqqq] (- 3.3, 0.3) node {$g_1^1$};
\draw[color=sexdts] (-1.75,1) node {$g_1^2$};
\draw[color=sexdts] (-1.4,1.3) node {$g_2$};

\draw[color=rvwvcq] (5.5, 5) node {$\nabla f_2(x)$};
\draw[color=sexdts] (3.3, 5) node {$g_2^2$};
\draw[color=ffqqqq] (6.3, 3.5) node {$g_2^1$};
\draw[color=rvwvcq] (-1, 1.9) node {$g$};
\draw[color=ffqqqq] (-0.2, 1.45) node {$g_1$};
\end{scriptsize}
\end{tikzpicture}
\caption{Illustration of the solutions of Subproblems~\eqref{subproblem2} and~\eqref{subproblem3}.\label{fig:SGM}} 
\end{figure}

As we mention in the introduction, let $S(x, \lambda) =\textstyle \sum_{i = 1}^m \lambda_i f_i(x)$ denote the weighted true function and $\nabla_x S(x, \lambda) = \textstyle \sum_{i = 1}^m \lambda_i \nabla f_i(x)$ the corresponding gradient.

When $m=1$, recall that is classical to assume that the stochastic gradients are unbiased estimates of the corresponding true gradients.
In the MOO case ($m>1$), even if $g_i(x, w)$ are unbiased estimates of $\nabla f_i(x) $ for all $ i = 1, \ldots, m$, the stochastic multi-gradient $g(x, w)$ resulting from solving~\eqref{subproblem3} is a biased estimate of~$\mathbb{E}_{w}[\nabla_x S(x, \lambda^g)]$, where $\lambda^g$ are the convex combination coefficients associated with $g(x, w)$, or even of the true multi-gradient $\nabla_x S(x, \lambda)$, where $\lambda$ are now the coefficients that result from solving~\eqref{subproblem2} with true gradients $\nabla f_i(x)$, $i=1,\ldots,m$. Basically, the solution of the QP~\eqref{subproblem3} acts as a mapping on the unbiased stochastic gradients~$g_i(x, w), i = 1, \ldots, m$, introducing biasedness in the mapped outcome (the stochastic multi-gradient~$g(x, w))$.

Let us observe the amount of biasedness by looking at the norm of the expected error \linebreak $\|\mathbb{E}_{w}[g(x, w) -\nabla_xS(x, \lambda^g)]\|$ in an experiment with $n=4$ and $m=2$, where each objective function is a finite sum of the form~(\ref{ml_empirical_risk_moo}) with $N=3,000$. Given two $4$-dimensional vectors $v_1 = (-31.22,	-26.57,	-17.20,	18.77)$ and $v_2 = (-7.50,	15.62,	-16.69,	15.31)$, we randomly sampled two sets of $3,000$ four-dimensional vectors from normal distributions with means~$v_1, v_2$, respectively, and covariance~$0.2I_4$, where $I_4$ is a four-dimensional identity matrix.
These two sets of $3,000$ vectors are the component gradient vectors of the two finite sum objectives.
A batch size specifies for each objective how many samples in a batch are uniformly drawn from the set of the $3,000$~stochastic gradients. For each batch size, we drew $10,000$ batches of $g_1, g_2$ and took means over $g(x, w) - \nabla_x S(x, \lambda^g)$. Figure~\ref{fig:biasedness} (a) confirms the existence of biasedness in the stochastic multi-gradient $g(x, w)$ as an approximation to $\mathbb{E}_{w}[\nabla_x S(x, \lambda^g)]$. It is observed that the biasedness does indeed decrease as the batch size increases, and that it eventually vanishes in the full batch regime (where Subproblems~\eqref{subproblem2} and~\eqref{subproblem3} become identical).

\begin{figure}[H]
   \centering
   \subfloat[][Using~$\lambda^g$.]{\includegraphics[width=.47\textwidth, height =.4\textwidth]{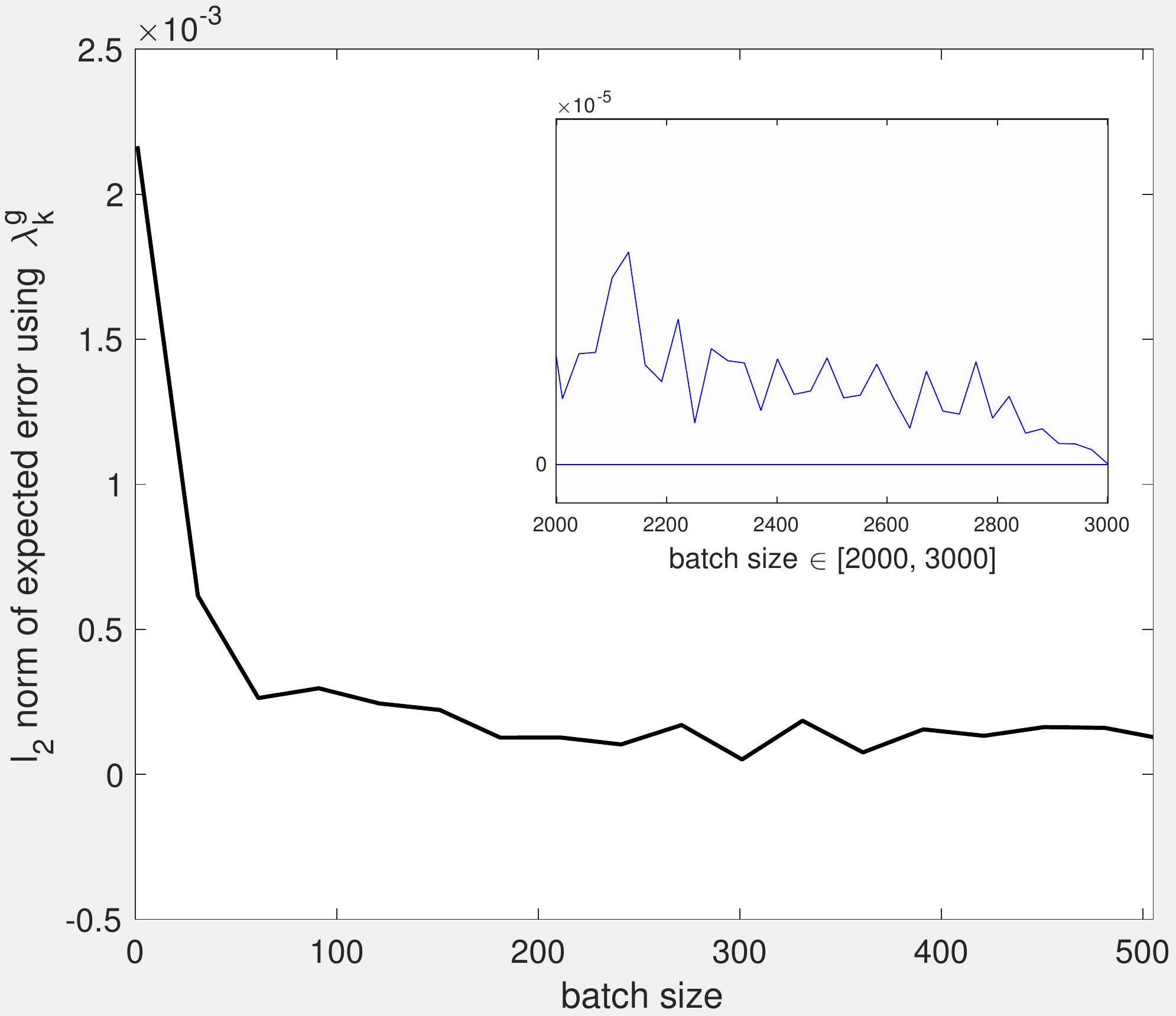}}\quad
   \subfloat[][Using~$\lambda$.]{\includegraphics[width=.47\textwidth, height =.4\textwidth]{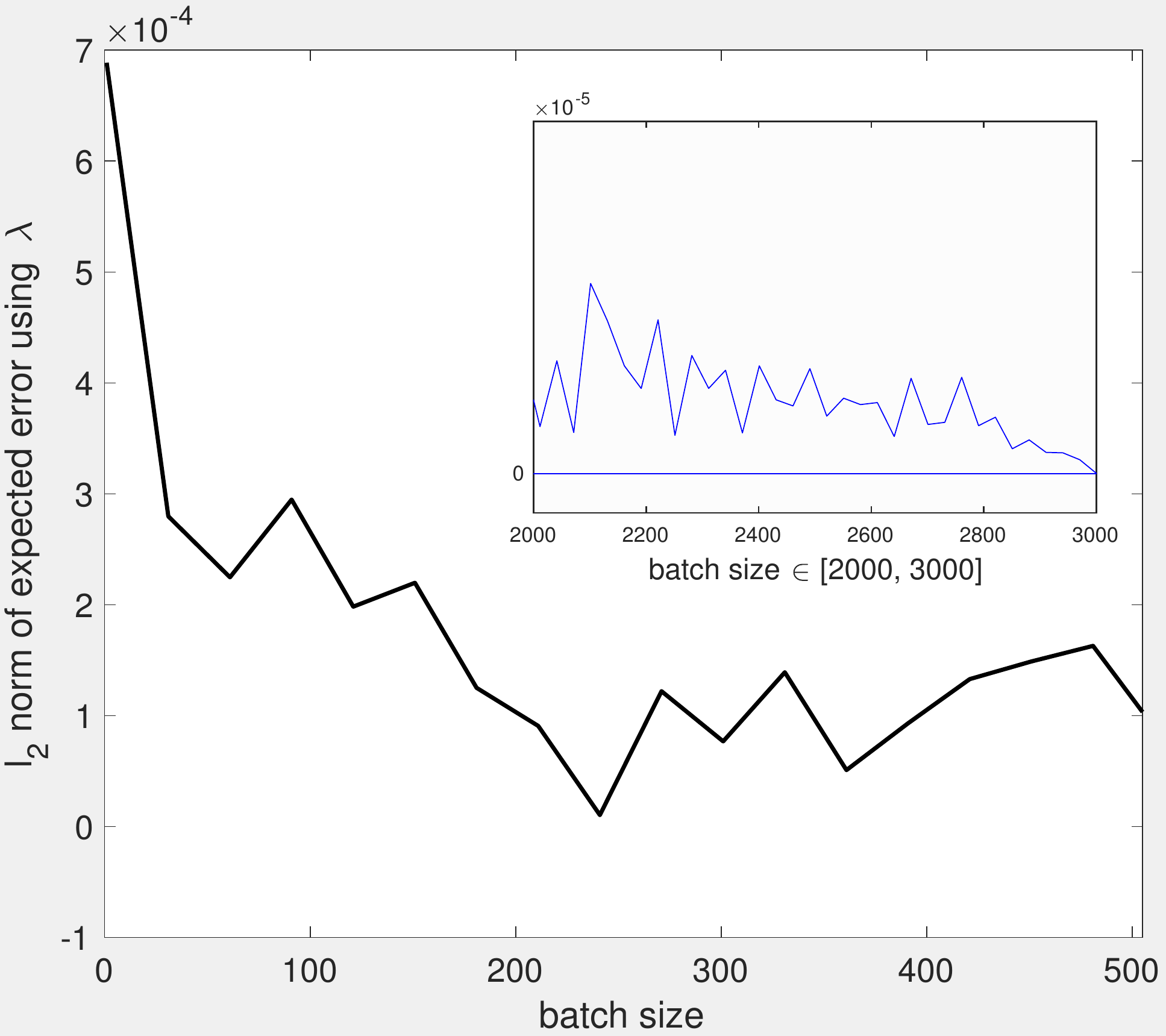}}
   \caption{Biasedness decreases as the batch size increases: $m=2$, $n=4$, and $N = 3000$.}
   \label{fig:biasedness}
\end{figure}


Biasedness is also present when we look at the norm of the expected error~$\|\mathbb{E}_{w}[g(x, w)] -\nabla_xS(x, \lambda)\|$, using the true coefficients~$\lambda$. In the same setting of the previous experiment,  Figure~\ref{fig:biasedness} (b) shows that biasedness still exists, although in a smaller quantity than when using~$\lambda^g$.


\section{Convergence rates for the stochastic multi-gradient method}
\label{convergence_analysis_smg}

In this section, the convergence theory of the simple stochastic gradient method is extended to stochastic MOO. We prove sublinear convergence rates of the order of $1/k$ and $1/\sqrt{k}$ for strongly convex and convex cases, respectively, when approaching a point in the Pareto front.
Let us start by formulating the assumptions that are common to both cases. First of all, as in the case $m=1$, we assume that all the objective functions in problem~\eqref{stochastic_moo} are sufficiently smooth.

\begin{assumption} \label{ass:Lipschitz}(\textbf{Lipschitz continuous gradients})
All objective  functions $f_i: \mathbb{R}^n \to \mathbb{R}$ are continuously differentiable with gradients $\nabla f_i$ Lipschitz continuous with Lipschitz constants $L_i > 0$, $i = 1, \ldots, m$, i.e.,
$
         \|\nabla f_i(x) - \nabla f_i(\bar{x}) \| \;\leq\; L_i \| x - \bar{x} \|,  \forall (x, \bar{x}) \in \mathbb{R}^n \times \mathbb{R}^n.
$
\end{assumption}

Assumption~\ref{ass:Lipschitz} implies smoothness of the weighted true function $S(x, \lambda)= \textstyle \sum_{i=1}^m \lambda_i f_i(x)$. In fact, given any $\lambda \in \Delta^m$, the weighted true function $S(x,\lambda)$ has Lipschitz continuous gradients in $x$ with constant $L = \textstyle \max_{{1 \leq i \leq m}}\{L_i\}$, i.e.,
\begin{equation}
    \| \nabla_x S(x, \lambda) - \nabla_x S(\bar{x}, \lambda) \| \;\leq\; L\| x - \bar{x} \|, \quad \forall (x, \bar{x}) \;\in\; \mathbb{R}^n \times \mathbb{R}^n.
    \label{ass1:Lipschitz}
\end{equation}

We will use $\mathbb{E}_{w_k}[\cdot]$ to denote the expected value taken with respect to $w_k$. Notice that $x_{k+1}$ is a random variable depending on $w_k$ whereas $x_k$ does not.

Now we propose our assumptions on the amount of biasedness and variance of the stochastic multi-gradient~$g(x_k,w_k)$. As commonly seen in the literature of the standard stochastic gradient method~\cite{LBottou_FECurtis_JNocedal_2018, ANemirovski_2009}, we assume that the individual stochastic gradients~$g_i(x_k, w_k), i = 1, \ldots, m$, are unbiased estimates of the corresponding true gradients and that their variance is bounded by the size of these gradients (Assumptions~(a) and~(c) below). However, an assumption is also needed to bound the amount of biasedness
of the stochastic multi-gradient in terms of the true gradient norm and the stepsize~$\alpha_k$ (Assumption~(b) below).

\begin{assumption}
\label{ass:unbiasedness}
For all objective functions $f_i$, $i = 1, \ldots, m$, and iterates $k \in \mathbb{N}$, the individual stochastic gradients~$g_i(x_k, w_k)$ satisfy the following:
\begin{enumerate}
    \item[(a)]\textbf{(Unbiasedness)} $\mathbb{E}_{w_k}[g_i(x_k, w_k)] \;=\; \nabla f_i(x_k)$.
    \item[(b)]\textbf{(Bound on the first moment)}
    There exist positive scalars $C_i > 0$ and $\hat{C}_i > 0$ such that
\begin{equation}
\label{ass2::bounded_diff_gi}
    \mathbb{E}_{w_k}\left[\left\|g_i(x_k, w_k) - \nabla f_i(x_k)\right\| \right] \;\leq\; \alpha_k \left(C_i + \hat{C}_i \|\nabla f_i(x_k)\| \right).
\end{equation}
    \item[(c)](\textbf{Bound on the second moment})
There exist positive scalars $G_i > 0$ and $\hat{G}_i > 0$ such that
\begin{equation*}
\label{ass2::bounded_variance}
\begin{split}
\mathbb{V}_{w_k}[g_i(x_k, w_k)] \;\leq\; G^2_i + \hat{G}^2_i\|\nabla f_i(x_k)\|^2.
\end{split}
\end{equation*}
\end{enumerate}
\end{assumption}

In fact, based on inequality \eqref{ass2::bounded_diff_gi}, one can derive an upper bound for the biasedness of the stochastic multi-gradient
\begin{equation*}
\begin{split}
\left\|\mathbb{E}_{w_k}[g(x_k, w_k) -\nabla_x S(x_k, \lambda^g_k)]\right\| \;\leq\; & \mathbb{E}_{w_k}\left[\left\|g(x_k, w_k) -\nabla_xS(x_k, \lambda^g_k)\right\|\right] \\
\;=\; & \mathbb{E}_{w_k}\left[\left \|\sum_{i=1}^m (\lambda^g_k)_i(g_i(x_k, w_k) - \nabla f_i(x_k)) \right\|\right] \\
\;\leq\; & \sum_{i =1}^m \mathbb{E}_{w_k}[\|g_i(x_k, w_k) - \nabla f_i(x_k)\|] \\
\;\leq\; & \alpha_k \left (\sum_{i=1}^m C_i + \sum_{i=1}^m \hat{C}_i \|\nabla f_i(x_k)\| \right),
\end{split}
\end{equation*}
where the first inequality results from Jensen's inequality in the context of probability theory. As a consequence, we have
\begin{equation}
\label{ass:sing_grad_diff_bound}
    \|\mathbb{E}_{w_k}[g(x_k, w_k) -\nabla_xS(x_k, \lambda^g_k)]\|  \;\leq\; \alpha_k\left(M_1 + M_F \sum_{i=1}^m \|\nabla f_i(x_k)\| \right)
\end{equation}
with $M_1 = \sum_{i=1}^m C_i$ and $M_F = \max_{1\leq i \leq m} \hat{C}_i$.
Note that we could have imposed directly the assumption
\[
    \|\mathbb{E}_{w_k}[g(x_k, w_k) -\nabla_xS(x_k, \lambda^g_k)]\|  \;\leq\; \alpha_k\left(M_1 + M_F \left\|\mathbb{E}_{w_k}[\nabla_xS(x_k, \lambda^g_k)]\right\|\right),
\]
from which then~\eqref{ass:sing_grad_diff_bound} would have easily followed.
However we will see later that we will also need the more general version stated in Assumption \ref{ass:unbiasedness}~(b).

Using Assumptions~\ref{ass:unbiasedness} (a) and (c), we can generalize the bound on the variance of the individual stochastic gradients $g_i(x_k, w_k)$ to the stochastic multi-gradient $g(x_k, w_k)$. In fact we first note that
\begin{equation*}
\begin{split}
\mathbb{E}_{w_k}[\|g_i(x_k, w_k)\|^2] \;=\; & \mathbb{V}_{w_k}[g_i(x_k, w_k)] + \big\|\mathbb{E}_{w_k}[g_i(x_k, w_k)]\big\|^2\\ \vspace{0.5ex}
 \;\leq\; & G^2_i + (\hat{G}^2_i + 1)\|\nabla f_i(x_k)\|^2,
\end{split}
\end{equation*}
from which we then obtain
\begin{equation*}
\label{ass:second_moment}
\begin{split}
\mathbb{E}_{w_k}[\|g(x_k, w_k)\|^2] \;=\; &\mathbb{E}_{w_k}\left[\left\|\sum_{i=1}^m (\lambda_k^g)_i g_i(x_k, w_k)\right\|^2\right] \\
 \;\leq\; &
\mathbb{E}_{w_k}\left[m\sum_{i=1}^m\|g_i(x_k, w_k)\|^2\right] \\
\;\leq\; & m \sum_{i=1}^m \left(G^2_i + (\hat{G}^2_i + 1)\|\nabla f_i(x_k)\|^2\right)\\
\;=\; & G^2 + G_V^2\sum_{i=1}^m \|\nabla f_i(x_k)\|^2
\end{split}
\end{equation*}
with $G^2 = m\sum_{i=1}^m G_i^2$ and $G_V^2 = m \max_{1\leq i \leq m} (\hat{G}_i^2 + 1)$.  Note that the obtained inequality is consistent with imposing directly a bound of the form
\begin{equation*}
\mathbb{E}_{w_k}\left[\|g(x_k, w_k)\|^2\right] \;\leq\; G^2 + G_V^2 \left\|\mathbb{E}_{w_k}[\nabla_x S(x_k, \lambda^g_k)]\right\|^2
\end{equation*}
from the fact that~$\textstyle\|\mathbb{E}_{w_k}[\nabla_x S(x_k, \lambda^g_k)]\|^2 \leq m\sum_{i=1}^m \|\nabla f_i(x_k)\|^2$.

We will need the iterates to lie in a bounded set, one of the reasons being the need to bound the norm of the true gradients. We can achieve this by asking $\mathcal{X}$ to
be a bounded set (in addition to being closed and convex).

\begin{assumption}
\label{ass:limit_pts}
The feasible region $\mathcal{X} \subset \mathbb{R}^n$ is a bounded set.
\end{assumption}

The above assumption implies the existence of an upper bound on the diameter of the feasible region, i.e., there exists a positive constant $\Theta$ such that
\begin{equation}
\label{compact_region_bound}
\displaystyle \max_{x, y \in \mathcal{X}} \| x - y\| \;\leq\; \textstyle \Theta \;<\; \infty.
\end{equation}
Note that from Assumption \ref{ass:Lipschitz} and \eqref{compact_region_bound}, the norm of the true gradient of each objective function is bounded, i.e., $\|\nabla f_i(x) \| \leq M_{\nabla} + L\Theta$, for $i = 1, \ldots, m$, and any $x \in \mathcal{X}$, where $M_{\nabla}$ denotes the largest of the norms of the $\nabla f_i$ at an arbitrary point of~$\mathcal{X}$.
For conciseness, denote $L_{\nabla S} = M_1 + m M_F(M_{\nabla}+L\Theta)$ and $L^2_g = G^2 + mG_V^2(M_{\nabla}+L\Theta)^2$. Hence, we have
\begin{equation}
\label{ineq:theorem_bound1}
\left\|\mathbb{E}_{w_k}\left[g(x_k, w_k) - \nabla_x S(x_k, \lambda^g_k)\right]\right\| \;\leq\; \alpha_k L_{\nabla S}
\end{equation}
and
\begin{equation}
\label{ineq:theorem_bound2}
\mathbb{E}_{w_k}\left[\|g(x_k, w_k)\|^2\right] \;\leq\; L_g^2.
\end{equation}

Lastly, we need to bound the sensitivity of the solution of the Subproblem~\eqref{subproblem2},
a result that follows locally from classical sensitivity theory but that we assume globally too.

\begin{assumption}
\label{ass:subprob_lip}
(\textbf{Subproblem Lipschitz continuity}) The optimal solution of Subproblem~\eqref{subproblem2} is a Lipschitz continuous function of the parameters $\{\nabla f_i(x), {1 \leq i \leq m}\}$, i.e., there exists a scalar  $\beta > 0$ such that
\begin{equation*}
 \|\lambda_k - \lambda_s\| \;\leq\; \beta \left \|\left[(\nabla f_1(x_k) - \nabla f_1(x_s))^\top, \ldots, (\nabla f_m(x_k) - \nabla f_m(x_s))^\top \right]\right \|.
 \end{equation*}
\end{assumption}

The above assumption states that the function $\lambda = \lambda (v_1, \ldots, v_m)$, where $v_i \in \mathbb{R}^n$, $i = 1, \ldots, m$, is Lipschitz continuous on $(v_1, \ldots, v_m)$. Recall that the only change from Subproblem~\eqref{subproblem2} to~\eqref{subproblem3} is replacing the true gradients by the corresponding stochastic gradients. As a consequence, the optimal solutions of Subproblems~\eqref{subproblem2} and~\eqref{subproblem3} satisfy
\begin{equation}
\label{subprob_lip_lambda1}
 \begin{split}
\mathbb{E}_{w_k}[\|\lambda^g_k - \lambda_k\| ] \;\leq\; & \beta \mathbb{E}_{w_k} \Big[\big\|[(g_1(x_k, w_k) - \nabla f_1(x_k))^\top, \ldots, (g_m(x_k, w_k) - \nabla f_m(x_k))^\top ]\big\| \Big] \\
\;\leq\; & \beta\sum_{i=1}^m \mathbb{E}_{w_k}\big[\|g_i(x_k, w_k) - \nabla f_i(x_k)\| \big] \\
\;\leq\; & \alpha_k (\beta L_{\nabla S}),
\end{split}
 \end{equation}
where $L_{\nabla S}$ is the constant defined in~\eqref{ineq:theorem_bound1}. Since $\nabla_x S(x, \lambda)$ is a linear function of $\lambda$,
\begin{equation*}
    \left\|\nabla_x S(x_k, \lambda_k^g) - \nabla_x S(x_k, \lambda_k)\right\| \;\leq\; M_{S} \left\|\lambda^g_k - \lambda_k\right\|,
\end{equation*}
with $M_S = \sqrt{mn} (M_{\nabla} + L \Theta)$.
By taking expectation over $w_k$ and using \eqref{subprob_lip_lambda1}, one obtains
\begin{equation}
\label{subprob_lip_lambda2}
    \mathbb{E}_{w_k}\left[\left\|\nabla_x S(x_k, \lambda_k^g) - \nabla_x S(x_k, \lambda_k)\right\|\right] \;\leq\; \alpha_k (\beta L_{\nabla S} M_S ).
\end{equation}

\subsection{The strongly convex case}
\label{subsec:SC}

Strongly convexity is the most widely studied setting in stochastic gradient methods. In the context of MOO we impose it in all individual functions.

\begin{assumption}
\label{ass:strongconv}
(\textbf{Strong convexity}) All objective functions $f_i: \mathbb{R}^n \to \mathbb{R}$ are strongly convex, i.e., for all $i = 1, \ldots, m$, there exists a scalar $c_i > 0$ such that
\begin{equation*}
f_i(\bar{x}) \;\geq\; f_i(x) + \nabla f_i(x)^\top(\bar{x} - x) + \frac{c_i}{2} \| \bar{x} - x \|^2, \quad \forall (x,\bar{x}) \in \mathbb{R}^n \times \mathbb{R}^n.
\end{equation*}
\end{assumption}

Under this assumption all the individual functions have a unique minimizer in~$\mathcal{X}$ that is also a Pareto minimizer. We also conclude that the weighted function $S(x, \lambda)$ is strongly convex with constant $c = \textstyle \min_{1 \leq i \leq m}\{c_i\}$, i.e.,
\begin{equation}
    S(\bar{x}, \lambda) \;\geq\; S(x, \lambda) + \nabla_x S(x, \lambda)^\top(\bar{x} - x) + \frac{c}{2} \| \bar{x} -x \|^2, \quad \forall (x,\bar{x})\in \mathbb{R}^n \times \mathbb{R}^n.
    \label{ass:strong_cov}
\end{equation}

We are finally ready to prove a convergence rate under strong convexity, showing that a certain weighted function has the potential to decay sublinearly at the rate of~$1/k$, as it happens in the stochastic gradient method for $m=1$.
We will use $\mathbb{E}[\cdot]$ to denote the expected value taken with respect to the joint distribution of $\{w_k, k \in \mathbb{N}\}$.
The stepsize choice~$\alpha_k$ is of diminishing type, in other words it obeys
\begin{equation*}
\label{diminishing_step}
    \sum_{k = 1}^\infty \alpha_k \;=\; \infty \quad \text{and} \quad \sum_{k = 1}^\infty \alpha_k^2 \;<\; \infty.
\end{equation*}

\begin{theorem} (\textbf{sublinear convergence rate under strong convexity})
\label{th:strongcov}
Let Assumptions~\ref{ass:Lipschitz}--\ref{ass:strongconv} hold and $x_*$ be any point in $\mathcal{X}$. Consider a diminishing step size sequence
$\alpha_k = \frac{2}{c(k+1)}$. The sequence of iterates generated by Algorithm~\ref{alg:SMG} satisfies
 \begin{equation*}
 \label{theorem1_result}
  \min_{s = 1, \ldots, k} \mathbb{E}[ S(x_s, \lambda_s)] -  \mathbb{E}[S({x}_*, \bar{\lambda}_k)] \;\leq\; \frac{2L_g^2 + 4\Theta(L_{\nabla S} + \beta L_{\nabla S} M_S)}{c(k+1)},
 \end{equation*}
 where  $\bar{\lambda}_k = \sum_{s=1}^k \frac{s}{\sum_{s = 1}^k s} \lambda_s \in \Delta^m$.
\end{theorem}

\begin{proof} For any $k \in \mathbb{N}$, considering that the projection operation is non-expansive, one can write
 \begin{equation}
    \begin{split}
     \mathbb{E}_{w_k}[\| x_{k+1} - x_*\|^2]  \;=\; &\mathbb{E}_{w_k}[\|P_{\mathcal{X}}( x_{k} - \alpha_kg(x_k, w_k)) - x_*\|^2]\\
     \;\leq\; & \mathbb{E}_{w_k}[\| x_{k} - \alpha_kg(x_k, w_k) - x_*\|^2]\\
     \;=\; & \| x_{k}- x_*\|^2 + \alpha_k^2 \mathbb{E}_{w_k}[\| g(x_k, w_k)\|^2] \\
     & - 2\alpha_k \mathbb{E}_{w_k}[g(x_k, w_k)]^\top (x_k - x_*).
     \label{proof_ineq_step0}
    \end{split}
  \end{equation}
 Adding the null term $2\alpha_k(\mathbb{E}_{w_k}[\nabla_x S(x_k, \lambda_k^g)] - \mathbb{E}_{w_k}[\nabla_x S(x_k, \lambda_k^g)] + \nabla_x S(x_k, \lambda_k) - \nabla_x S(x_k, \lambda_k))$ to the right-hand side yields
  \begin{equation}
      \begin{split}
     \mathbb{E}_{w_k}[\| x_{k+1} - x_*\|^2]  \;\leq\; & \| x_{k}- x_*\|^2 + \alpha_k^2\mathbb{E}_{w_k}[\| g(x_k, w_k)\|^2] - 2\alpha_k\nabla_x S(x_k, \lambda_k)^\top (x_k - x_*) \\
    & + 2\alpha_k \| \mathbb{E}_{w_k}[g(x_k, w_k)  - \nabla_x S(x_k, \lambda^g_k)]\| \| x_k - x_*\| \\
    & + 2\alpha_k \|\mathbb{E}_{w_k}[\nabla_x S(x_k, \lambda^g_k) - \nabla_x S(x_k, \lambda_k)]\| \|x_k - x_*\|.
    \end{split}
    \label{proof_ineq_step1}
\end{equation}
Choosing $\lambda = \lambda_k$, $x = x_k$, and $\bar{x} = x_*$ in inequality \eqref{ass:strong_cov}, one has
 \begin{equation}
 \label{theorm:strongconv}
 \nabla_x S(x_k, \lambda_k)^\top(x_k - x_*)  \;\geq\; S(x_k, \lambda_k) - S(x_*, \lambda_k) + \frac{c}{2}\| x_k - x_*\|^2.
 \end{equation}
 Then, plugging inequalities \eqref{ineq:theorem_bound1}, \eqref{ineq:theorem_bound2}, \eqref{subprob_lip_lambda2}, and \eqref{theorm:strongconv} into \eqref{proof_ineq_step1}, we obtain
 \begin{align*}
  \mathbb{E}_{w_k} [\| x_{k+1} - x_*\|^2]  \;\leq\; & \;  (1-\alpha_k c) \| x_{k}- x_*\|^2 + \alpha_k^2 (L_g^2 + 2\Theta(L_{\nabla S} + \beta L_{\nabla S} M_S)) \\
  & - 2\alpha_k \mathbb{E}_{w_k} [S(x_k, \lambda_k) - S(x_*, \lambda_k)].
 \end{align*}

For simplicity denote $M = L_g^2 + 2\Theta(L_{\nabla S} + \beta L_{\nabla S} M_S)$. Using $\alpha_k = \frac{2}{c(k+1)}$, and rearranging the last inequality,
 \begin{align*}
\mathbb{E}_{w_k} [S(x_k, \lambda_k) - S(x_*, \lambda_k) ] \;\leq\; & \frac{(1-\alpha_k c) \| x_{k}- x_*\|^2 + \alpha_k^2 M - \mathbb{E}_{w_k} [\| x_{k+1} - x_*\|^2]}{2\alpha_k } \\
\;\leq\; & \frac{c(k-1)}{4} \| x_k - x_*\|^2 - \displaystyle \frac{c(k+1)}{4} \mathbb{E}_{w_k} [\| x_{k+1} - x_*\|^2] + \frac{M}{c(k+1)}.
\end{align*}
Now we replace $k$ by $s$ in the above inequality. Taking the total expectation, multiplying by $s$ on both sides, and summing over $s = 1, \ldots,k$ yields
 \begin{align*}
\sum_{s=1}^k s(\mathbb{E}[ S(x_s, \lambda_s)] - \mathbb{E}[S(x_*, \lambda_s)]) \;\leq\; & \sum_{s=1}^k \left(\frac{cs(s-1)}{4} \mathbb{E}[\left\lVert x_s - x_*\right\lVert^2 ] - \frac{cs(s+1)}{4} \mathbb{E} [\left\lVert x_{s+1} - x_*\right\lVert^2]\right)  \\
& + \sum_{s=1}^k\frac{s}{c(s+1)}M\\
\;\leq\; & - \frac{c}{4}k(k+1) \mathbb{E} [\left\lVert x_{k+1} - x_*\right\lVert^2] + \sum_{s=1}^k\frac{s}{c(s+1)} M\\
\;\leq\; & \frac{k}{c} M.
\end{align*}
Dividing both sides of the last inequality by $\textstyle \sum_{s = 1}^k s$ gives us
\begin{equation}
\label{proof_ineq_step2}
\begin{split}
\frac{\sum_{s=1}^k s \mathbb{E}[ S(x_s, \lambda_s)] - \sum_{s=1}^k s \mathbb{E}[S(x_*, \lambda_s)] }{\sum_{s = 1}^k s}
& \;\leq \; \frac{kM}{c\sum_{s = 1}^k s}  \;\leq\; ~\frac{2M}{c(k+1)}.
 \end{split}
 \end{equation}
The left-hand side is taken care as follows
\begin{equation}
\label{proof_ineq_step3}
 \begin{split}
   \min_{s = 1, \ldots, k} \mathbb{E}[ S(x_s, \lambda_s)] - \mathbb{E}[S(x_*, \bar{\lambda}_k)] \;\leq\; \sum_{s=1}^k \frac{s}{\sum_{s = 1}^k s}\mathbb{E}[S(x_s, \lambda_s)]  - \sum_{s=1}^k \frac{s}{\sum_{s = 1}^k s} \mathbb{E}[S(x_*, \lambda_s)],
 \end{split}
 \end{equation}
where $\bar{\lambda}_k = \sum_{s=1}^k \frac{s}{\sum_{s = 1}^k s} \lambda_s$. The proof is finally completed by combining \eqref{proof_ineq_step2} and \eqref{proof_ineq_step3}.
\end{proof}

Since the sequence $\{\lambda_k\}_{k \in \mathbb{N}}$ generated by Algorithm \ref{alg:SMG} is bounded, it has a limit point $\lambda_*$. Assume that the whole sequence $\{\lambda_k\}_{k \in \mathbb{N}}$ converges to $\lambda_*$. Let $x_*$ be the unique minimizer of $S(x, \lambda_*)$.  Then, $x_*$ is a Pareto minimizer associated with $\lambda_*$. Since $\bar{\lambda}_k$ is also converging to $\lambda_*$, $\mathbb{E}[S({x}_*, \bar{\lambda}_k)] $ converges to $\mathbb{E}[S({x}_*, \lambda_*)] $. Hence, Theorem~\ref{th:strongcov} states that $\textstyle \min_{1 \leq s \leq k} \mathbb{E}[ S(x_s, \lambda_s)] $ converges to $\mathbb{E}[S({x}_*, \lambda_*)] $.  The result of Theorem~\ref{th:strongcov} indicates that the approximate rate of such convergence is $1/k$.
Rigorously speaking, since we do not have $\lambda_*$ on the left-hand side but rather $\bar{\lambda}_k$, such left-hand side is not even guaranteed to be positive. The difficulty comes from the fact that $\lambda_*$ is only defined at convergence and the multi-gradient method cannot anticipate which optimal weights are being approached, or in equivalent words which weighted function is being minimized at the end.
Such a difficulty is resolved if we assume that $\lambda_k$ approximates well the role of $\lambda_*$ at the Pareto front.

\begin{assumption}
\label{ass::stronger_conv}
Let $x_*$ be the Pareto minimizer defined above.
For any $x_k$, one has
\begin{equation*}
\nabla_x S(x_*, \lambda_k)^\top (x_k - x_*) \;\geq\; 0.
\end{equation*}
\end{assumption}

In fact notice that $\nabla_x S(x_*, \lambda_*) = 0 $ holds according to the Pareto stationarity condition~\eqref{pareto_stationarity2}, and thus this
assumption
would hold with $\lambda_k$ replaced by $\lambda_*$.

A well-known equivalent condition to~\eqref{ass:strong_cov}
is
\begin{equation*}
(\nabla_x S(x, \lambda) - \nabla_x S(\bar{x}, \lambda) )^\top (x - \bar{x}) \;\geq\; c\|x - \bar{x}\|^2, \quad \forall (x, \bar{x}) \;\in\; \mathbb{R}^n \times \mathbb{R}^n.
\end{equation*}
Choosing $x = x_k$, $\bar{x} = x_*$, and $\lambda = \lambda_k$ in the above inequality and using Assumption~\ref{ass::stronger_conv} leads to
\begin{equation}
\label{ass:strong_cov2}
\nabla_x S(x_k, \lambda_k)^\top (x_k - x_*) \;\geq\; c\|x_k - x_*\|^2,
\end{equation}
based on which one can derive a stronger convergence result\footnote{Let us see how Assumption~\ref{ass::stronger_conv} relates to Assumption~H5 used in~\cite{MQuentin_PFabrice_JADesideri_2018}. These authors have made the strong assumption that the noisy values satisfy a.s. $f_i(x,w) - f_i(x,w) \geq C_i \| x-x^\perp\|^2$ for all $x$, where $x^\perp$ is the point in $\mathcal{P}$ closest to~$x$ (and $C_i$ a positive constant). From here they easily deduce from the convexity of the individual functions $f_i$ that $\mathbb{E}_{w_k}[ g(x_k,w_k)]^\top (x_k-x_k^\perp) \geq 0$, which then leads to establishing that $\mathbb{E}[\| x_{k}- x_k^\perp\|^2] = \mathcal{O}(1/k)$.

Notice that $\mathbb{E}_{w_k}[ g(x_k,w_k)]^\top (x_k-x_k^\perp) \geq 0$ would also result from~\eqref{ass:strong_cov2} (with $x_*$ replaced by $x_k^\perp$) if $g(x_k,w_k)$ was an unbiased estimator of $\nabla_x S(x_k, \lambda_k)$.}.

\begin{theorem}
Let Assumptions~\ref{ass:Lipschitz}--\ref{ass::stronger_conv} hold and $x_*$ be the Pareto minimizer corresponding to the limit point $\lambda_*$ of the sequence $\{ \lambda_k \}$.
Consider a diminishing step size sequence
$\alpha_k = \frac{\gamma}{k}$ where $\gamma \geq \frac{1}{2c}$ is a positive constant. The sequence of iterates generated by Algorithm~\ref{alg:SMG} satisfies
\begin{equation*}
\mathbb{E}[\| x_{k}- x_*\|^2] \;\leq\; \frac{\max \{2\gamma^2\bar{M}^2(2c\gamma - 1)^{-1}, \|x_0 - x_*\|^2\}}{k},
\end{equation*}
and
\begin{equation*}
\begin{split}
\mathbb{E}[S(x_k, \lambda_*)] - \mathbb{E}[S(x_*, \lambda_*)] \;\leq\; \frac{(L/2)\max \{2\gamma^2\bar{M}(2c\gamma - 1)^{-1}, \|x_0 - x_*\|^2\}}{k}
\end{split}
\end{equation*}
where  $\bar{M} = L_g^2 + 2\Theta(L_{\nabla S} + \beta L_{\nabla S} M_{S} )$.
\end{theorem}

\begin{proof}
Similarly to the proof of Theorem~\ref{th:strongcov}, from \eqref{proof_ineq_step0} to \eqref{proof_ineq_step1}, but using \eqref{ass:strong_cov2} instead of \eqref{theorm:strongconv}, one has
\begin{equation*}
    \mathbb{E}_{w_k}[\| x_{k+1} - x_*\|^2]  \;\leq\; (1 -  2\alpha_kc)\| x_{k}- x_*\|^2 + \alpha_k^2 (L_g^2 + 2\Theta(L_{\nabla S} + \beta L_{\nabla S} M_{S} )).
  \end{equation*}
Taking total expectation on both sides leads to
\begin{equation*}
\mathbb{E}[\| x_{k+1} - x_*\|^2]   \;\leq\;   (1 -  2\alpha_kc)\mathbb{E}[\| x_{k}- x_*\|^2] + \alpha_k^2 \bar{M}.
\end{equation*}
Using $\alpha_k = \gamma/k$ with $\gamma > 1/(2c)$ and
an induction argument (see~\cite[Eq.~(2.9) and (2.10)]{ANemirovski_2009}) would lead us to
\begin{equation*}
\mathbb{E}[\| x_{k}- x_*\|^2] \;\leq\; \frac{\max \{2\gamma^2\bar{M}(2c\gamma - 1)^{-1}, \|x_0 - x_*\|^2\}}{k}.
\end{equation*}
Finally, from an expansion using the Lipschitz continuity of $\nabla_x S(\cdot,\lambda_*)$ (see~\eqref{ass1:Lipschitz}), one can also derive a sublinear rate in terms of the optimality gap of the weighted function value
\begin{equation*}
\begin{split}
\mathbb{E}[S(x_k, \lambda_*)] - \mathbb{E}[S(x_*, \lambda_*)] \;\leq\; & \mathbb{E}[\nabla_x S(x_*, \lambda_*)]^\top(x_k - x_*) + \frac{L}{2} \mathbb{E}[\|x_k - x_*\|^2] \\
\;\leq\; & \frac{(L/2)\max \{2\gamma^2\bar{M}(2c\gamma - 1)^{-1}, \|x_0 - x_*\|^2\}}{k}.
\end{split}
\end{equation*}
\end{proof}

\subsection{The convex case}

In this section,  we relax the strong convexity assumption to convexity and derive a similar sublinear rate of $1/\sqrt{k}$ in terms of weighted function value. Similarly to the case $m=1$, we assume that the weighted functions attains a minimizer (which then also ensures that $\mathcal{P}$ is non empty).

\begin{assumption}
\label{ass:conv_and_obtainable}
All the objective functions~$f_i: \mathbb{R}^n \rightarrow \mathbb{R}$ are convex, $i=1,\ldots,m$. The convex function $S(\cdot, \lambda)$ attains a minimizer for any $\lambda \in \Delta^m$.
\end{assumption}

\begin{theorem} \label{th2:convex}
(\textbf{sublinear convergence rate under convexity}) Let Assumptions \ref{ass:Lipschitz}--\ref{ass:subprob_lip} and \ref{ass:conv_and_obtainable} hold and $x_*$ be any point in $\mathcal{X}$. Consider a diminishing step size sequence $\alpha_k = \textstyle \frac{\bar{\alpha}}{\sqrt{k}}$ where $\bar{\alpha}$ is any positive constant.
The sequence of iterates generated by Algorithm~\ref{alg:SMG} satisfies
\begin{equation*}
\label{theorem2_cov}
\min_{s = 1, \ldots, k} \mathbb{E}[S({{x}_s, \lambda}_s)] - \mathbb{E}[S(x_*, \bar{\lambda}_k)]
\;\leq\; \frac{\frac{\Theta^2 }{2\bar{\alpha}} + \bar{\alpha}(L^2_g + 2\Theta(L_{\nabla S}+  \beta L_{\nabla S}M_{S} ))}{\sqrt{k}},
\end{equation*}
 where $\bar{\lambda}_k = \textstyle \frac{1}{k}\sum_{s = 1}^k \lambda_s \in \Delta^m$.
\end{theorem}
\begin{proof}
See Appendix~\ref{append_proof_conv}.
\end{proof}

Similar comments and analysis as in the strongly convex (see the last part of Subsection~\ref{subsec:SC} after the proof of Theorem~\ref{th:strongcov}) could be here derived for the convex case. Note that if we let $\bar{\alpha} = c/2$, where $c$ is the strongly convex constant in Section 5.1, the constant part in the above theorem is larger by $\frac{\Theta^2}{2\bar{\alpha}}$ than in Theorem~\ref{th:strongcov}. Also, when comparing to the strongly convex case, not only the rate is worse in the convex case (as happens also when $m=1$) but also $\bar{\lambda}_k$ is now converging slower to~$\lambda_*$.


\subsection{Imposing a bound on the biasedness of the multi-gradient}
\label{Unbiasedness_verification}

Recall that from
\begin{equation}
\label{bias_verify_eq1}
\left\|\mathbb{E}_w\left [g(x, w) - \nabla_x S(x, \lambda^g)\right] \right\|
\;\leq\;  \sum_{i = 1}^m\mathbb{E}_{w}\left [\left \|g_i(x, w) - \nabla f_i(x) \right\|\right ],
\end{equation}
where $g_i(x, w)$ is the stochastic gradient at $x$ for the $i$-th objective function, and from Assumption~\ref{ass:unbiasedness}~(b), we derived a more general bound for the biasedness of the stochastic multi-gradient in~\eqref{ass:sing_grad_diff_bound}, whose right-hand side involves the stepsize~$\alpha_k$.  For simplicity, we will again omit the index~$k$ in the subsequent analysis.

 We will see that the right-hand side of~\eqref{bias_verify_eq1} can be always
 (approximately) bounded by a dynamic sampling strategy when calculating the stochastic gradients for each objective function. The idea is similar to mini-batch stochastic gradient, in the sense that by increasing the batch size the noise is reduced and thus more accurate gradient estimates are obtained.

Assumption~\ref{ass:unbiasedness} (a) states that $g_i(x, w), i = 1, \ldots, m$, are unbiased estimates of the corresponding true gradients. Let us assume that $g_i(x, w)$ is normally distributed with mean $\nabla f_i(x)$ and variance $\sigma^2_i$, i.e., $g_i(x, w) \sim \mathcal{N}(\nabla f_i(x), \sigma^2_iI_n)$, where $n$ is the dimension of $x$. For each objective function, one can obtain a more accurate stochastic gradient estimate by increasing the batch size. Let $b_i$ be the batch size for the $i$-th objective function and $ \bar{g}_i(x, w) = \frac{1}{b_i}\sum_{r = 1}^{b_i} g_i(x, w^r)$ be the corresponding batch stochastic gradient, where $\{w^r\}_{1\leq r\leq b_i}$ are realizations of $w$. Then, $G_i = g_i(x, w) - \nabla f_i(x)$ and $\bar{G}_i = \bar{g}_i(x, w) - \nabla f_i(x) , i = 1, \ldots, m$, are all random variables of mean~0. The relationship between ${G}_i$ and $\bar{G}_i$ is captured by (see~\cite{JEFreund_1962})
\begin{equation*}
\mathbb{V}_{w}[\bar{G}_i] \;\leq\; \frac{\mathbb{V}_{w}[{G}_i]}{b_i} \;\leq\; \frac{\sigma_i^2}{b_i}.
\end{equation*}
By the definition of variance $\mathbb{V}_{w}[G_i] = \mathbb{E}_{w}[\|G_i\|^2] - \|\mathbb{E}_{w}[G_i]\|^2$, $\|\mathbb{E}_{w}[\bar{G}_i]\| \leq \mathbb{E}_{w}[\|\bar{G}_i\|]$, and $\mathbb{E}_{w}[\bar{G}_i] = 0$, one has $\mathbb{V}_{w}[\|\bar{G}_i\|] \leq \mathbb{V}_{w}[\bar{G}_i] \leq \sigma_i^2/b_i$.
Then, replacing $g_i(x, w)$ in \eqref{bias_verify_eq1} by $\bar{g}_i(x, w)$, we have
\begin{equation*}
    \left\|\mathbb{E}_{w}\left [\bar{g}(x, w) - \nabla_x S(x, \lambda^g)\right] \right\| \;\leq\; ~ \sum_{i = 1}^m\mathbb{E}_{w} \left[  \|\bar{G}_i\|\right] \;\leq\; \sum_{i = 1}^m \frac{\sigma_i \sqrt{n}}{\sqrt{b_i}},
\end{equation*}
where the last inequality results from $\mathbb{E}[\|X\|] \leq \sigma \sqrt{n}$ for the random variable $X \sim \mathcal{N}(0, \sigma^2 I_n)$~\cite{VChandrasekaran_etal_2012}. Hence, one could enforce an inequality of the form $\sum_{i = 1}^m \frac{\sigma_i \sqrt{n}}{\sqrt{b_i}} \leq \alpha_k(M_1 + M_F\sum_{i=1}^m \|\nabla f_i(x)\|)$ to guarantee that~\eqref{ass:sing_grad_diff_bound} holds (of course replacing the size of the true gradients by some positive constant). Furthermore, to guarantee that the stronger bound~\eqref{ass2::bounded_diff_gi} holds, one can require $\mathbb{E}_{w}[\|\bar{G}_i\|] \leq \frac{\sigma_i\sqrt{n}}{\sqrt{b_i}} \leq \alpha(C_1 + \hat{C}_i\|\nabla f_i(x)\|)$ for each objective function. Intuitively, when smaller stepsizes are taken, the sample sizes $\{b_i\}_{1\leq i \leq m}$ should be increased or, correspondingly, smaller sample variances $\{\sigma_i\}_{1\leq i \leq m}$ are required.

\section{Pareto-front stochastic multi-gradient method}
\label{section_psmg_pmg}

The practical goal in many MOO problems is to calculate a good approximation of part of (or the entire) Pareto front, and for this purpose the SMG algorithm is insufficient as running it only yields a single Pareto stationary point. We will thus design a Pareto-Front Stochastic Multi-Gradient (PF-SMG) algorithm to obtain the complete Pareto front. The key idea of such an algorithm is to iteratively update a list of nondominated points which will render increasingly better approximations to the true Pareto front. The list is updated by essentially applying the SMG algorithm at some or all of its current points.
The PF-SMG algorithm begins with a list of (possibly random) starting points $\mathcal{L}_0$. At each iteration, before applying SMG and for sake of better performance, we first add to the list a certain number of perturbed points around each of the current ones. Then we apply a certain number of SMG steps multiple times at each point in the current list, adding each resulting final point to the list. Note that by applying SMG multiple times starting from the very same point, one obtains different final points due to stochasticity. The iteration is finished by removing all dominated points from the list. Since the new list ${\mathcal{L}}_{k+1}$ is obtained by removing dominated points from ${\mathcal{L}}_{k} \cup \mathcal{L}_k^{\text{new}}$, where $\mathcal{L}_k^{\text{new}}$ is the set of new points added to the current nondominated list~${\mathcal{L}}_{k}$, we only need to compare each point in $\mathcal{L}^{\text{new}}_k$ to the other points in ${\mathcal{L}}_{k} \cup \mathcal{L}^{\text{new}}_k$ in order to remove any dominated points~\cite{ALCustodio_etal_2011}.
The PF-SMG algorithm is formally described as follows.

{\linespread{1}\addtocounter{algorithm}{-1}
\renewcommand{\thealgorithm}{{2}} 
\begin{algorithm}[H] 
\caption{Pareto-Front Stochastic Multi-Gradient (PF-SMG) Algorithm}
\label{alg:PF_SMG}
\begin{algorithmic}[1] 
\par\vspace*{0.1cm}
\item
Generate a list of starting points $\mathcal{L}_0$. Select parameters $r, p, q \in \mathbb{N}$.
\item {\bf for} $k=0,1, \ldots$ {\bf do}
\item \quad\quad Set $\mathcal{L}_{k+1} = \mathcal{L}_k$.
\item \quad\quad {\bf for} each point $x$ in the list $\mathcal{L}_{k+1}$ {\bf do}
\item \quad\quad\quad\quad Add $r$ perturbed points to the list $\mathcal{L}_{k+1}$ from a neighborhood of $x$.
\item \quad\quad{\bf end for}
\item \quad\quad {\bf for} each point $x$ in the list $\mathcal{L}_{k+1}$ {\bf do}
\item \quad\quad\quad\quad {\bf for} $t = 1, \ldots, p$ {\bf do}
\item \quad \quad\quad\quad\quad Apply $q$ iterations of the SMG algorithm starting from~$x$.
 \item \quad \quad\quad\quad\quad Add the final output point $x_t$ to the list $\mathcal{L}_{k+1}$.
\item \quad \quad\quad\quad{\bf end for}
 \item \quad\quad {\bf end for}
\item \quad\quad Remove all the dominated points from $\mathcal{L}_{k+1}$:
{\bf for} each point $x$ in the list $\mathcal{L}_{k+1}$ {\bf do}
\item \quad\quad \quad  If $\exists~ y \in \mathcal{L}_{k+1}$ such that $F(y) < F(x)$ holds, remove $x$ from the list.
{\bf end for}
\item {\bf end for}
\par\vspace*{0.1cm}
\end{algorithmic}
\end{algorithm}
}

In order to evaluate the performance of the PF-SMG algorithm and have a good benchmark for comparison, we also introduce a Pareto-front version of the deterministic multi-gradient algorithm (acronym PF-MG). The PF-MG algorithm is exactly the same as the PF-SMG one except that one applies $q$ steps of  multi-gradient descent instead of stochastic multi-gradient to each point in Line~9. Also, $p$~is always equal to one in PF-MG.

\section{Numerical experiments}
\label{numerical_result}

\subsection{Parameter settings and metrics for comparison}

In our implementation, both PF-SMG and PF-MG algorithms use the same $30$ randomly generated starting points, i.e., $|\mathcal{L}_0| = 30$. In both cases we set $q = 2$. The step size is initialized differently according to the problem but always halved every $200$ iterations. Both algorithms are terminated when either the number of iterations exceeds~$1000$ or the number of points in the iterate list reaches~$1500$.

To avoid the size of the list growing too fast, we only generate the~$r$ perturbed points for pairs of points corresponding to the $m$ largest holes along the axes $f_i$,  $i = 1, \ldots, m$. More specifically, given the current list of nondominated points, their function values in terms of $f_i$ are first sorted in an increasing order, for $i = 1, \ldots, m$. Let $d_i^{j,j+1}$ be the distance between points $j$ and $j+1$ in $f_i$. Then, the pair of points corresponding to the largest hole along the axis $f_i$ is $(j_i, j_i+1)$, where $j_i = \argmax_{j} d_i^{j, j+1}$.

Given the fact that applying the SMG algorithm multiple times to the same point results in different output points, whereas this is not the case for multi-gradient descent, we take $p = 2$ for PF-SMG but let $p = 1$ for PF-MG. Then, we choose $r = 5$ for PF-SMG and $r = 10$ for PF-MG, such that the number of new points added to the list is the same at each iteration of the two algorithms.

To analyze the numerical results, we consider two types of widely-used metrics
to measure and compare Pareto fronts obtained from different algorithms, Purity~\cite{SBandyopadhya_SKPal_BAruna_2004} and Spread~\cite{KDeb_etal_2002}, whose mathematical formula are briefly recalled in Appendix~\ref{append_metrics_comp}. In what concerns the Spread metric, we use two variants, maximum size of the holes and point spread, respectively denoted by~$\Gamma$ and~$\Delta$.

\subsection{Supervised machine learning (logistic regression)}
\subsubsection{Problem description}

Fairness in machine learning has been extensively studied in the past few years~\cite{SBarocas_MHardt_ANarayanan_2017, MHardt_EPrice_NSrebro_2016, BWoodworth_etal_2017, MBZafar_etal_2017b, RZemel_etal_2013}. Given that many real-world prediction tasks such as credit scoring, hiring decisions, and criminal justice are naturally formulated as binary classification problems, we hence focus on the application of multi-objective optimization in fair binary classification. Generally, a classifier is said to be unfair if its prediction outcomes discriminate against people with sensitive attributes, such as gender and race. There are more than $20$ fairness measures proposed in the literature~\cite{SBarocas_MHardt_ANarayanan_2017, SVerma_JRubin_2018}, among which we propose to handle in this section \textit{overall accuracy equality} by a multi-objective formulation. By definition, a classifier satisfies overall accuracy equality if different subgroups have equal prediction accuracy~\cite{RBerk_etal_2018}. Our key idea is to construct a multi-objective problem where each objective corresponds to the prediction error of each subgroup. The resulting Pareto fronts help us identifying the existence of bias or unfairness and support fair decision-making. A more detailed study of the accuracy fairness trade-offs in machine learning using stochastic multi-objective optimization is given in~\cite{SLiu_LNVicente_2020}.


We have tested our idea on classical binary classification problems with the training data sets selected from LIBSVM~\cite{CCChang_CJLin_2011}. Each data set consists of feature vectors and labels for a number of data instances. The goal of binary classification is to fit the best prediction hyperplane in order to well classify the set of data instances into two groups. More precisely, for a given pair of feature vector $a$ and label $y$, we consider a separating hyperplane $x^\top a + b$ such that
\[\left\{
\begin{array}{ll}
    x^\top a + b  \;\geq\; 0 & \text{when }y \;=\;1, \\ \vspace{0.5ex}
     x^\top a + b  \;< \; 0 & \text{when }y \;=\;-1.\\
\end{array}
\right. \]
In our context, we evaluate the prediction loss using the smooth convex logistic function \linebreak $l(a, y; x, b) = \text{log}(1+\text{exp}(-y(x^\top a + b)))$, which leads us to a well-studied convex objective, i.e., logistic regression problem with the objective function being $\min_{x, b} \frac{1}{N}\sum_{j=1}^N \text{log}(1+\text{exp}(-y_j(x^\top a_j + b)))$, where $N$ is the size of training data. To avoid over-fitting, we need to add a regularization term $\frac{\lambda}{2}\|x\|^2$ to the objective function.

For the purpose of our study, we pick a feature of binary values and separate the given data set into two groups, with $J_1$ and $J_2$ as their index sets. An appropriate two-objective problem is formulated as $\min_{x, b}~(f_1(x, b), f_2(x, b))$, where
\begin{equation}
    f_i(x, b) = \frac{1}{|J_i|}\sum_{j \in J_i} \text{log}(1+e^{(-y_j(x^\top a_j + b))}) + \frac{\lambda_i}{2}\|x\|^2.
    \label{two_obj_LogRe}
\end{equation}

\subsubsection{Numerical results}

Our numerical results are constructed for four data sets: \textit{heart}, \textit{australian}, \textit{svmguide3}, and \textit{german.numer} \cite{CCChang_CJLin_2011}. First of all, we ran the single stochastic gradient (SG) algorithm with maximum $1000$ iterations to obtain a minimizer for the entire group, i.e., $J_1 \cup J_2$, based on which the classification accuracies for the whole group and for two groups separately are calculated. See Table~\ref{SGD_accuracy1} for a summary of the results.  ``Split'' indicates which feature is selected to split the whole group into two distinct groups.
The initial step size is also listed in the table.

\begin{table}[H]
\centering
\begin{tabular}{l|c|c|c|c|c|c}
\cline{1-7}
Data & $N$ & Step size & Split & Group 1 & Group 2 & Entire group \\ \cline{1-7}
\textit{heart}  & 270 & 0.1  & 2 & 93.1\%  & 80.3\%  & 84.4\%       \\ \cline{1-7}
\textit{australian} & 690 & 0.1 & 1 & 89.2\% & 85.9\%   & 87\%         \\ \cline{1-7}
\textit{svmguide3} & 1243 & 0.2 & 10 &  79.4\%   & 23.0\%   & 76.6\%       \\ \cline{1-7}
\textit{german.numer} & 1000 & 0.2 & 24 & 73.0\% & 79.4\% & 77\% \\ \cline{1-7}
\end{tabular}
\caption{Classification accuracy of the single SG.\label{SGD_accuracy1}}
\end{table}

One can easily observe that there exist obvious differences in term of the training accuracy between the two groups of data sets \textit{heart} and \textit{svmguide3}, whereas \textit{australian} and \textit{german.numer} have much smaller gaps. This means that classifying a new instance using the minimizer obtained from the single SG for the whole group might lead to large bias and poor accuracy. We then constructed a two-objective problem \eqref{two_obj_LogRe} with $\lambda_1 = \lambda_2 = 10^{-3}$ for the two groups of each data set. The PF-SMG algorithm has yielded the four Pareto fronts displayed in Figure~\ref{Pareto_fronts_LogRe}.
\begin{figure}[tb]
   \centering
   \subfloat[][\textit{heart}]{\includegraphics[width=.48\textwidth, height = 4.8 cm]{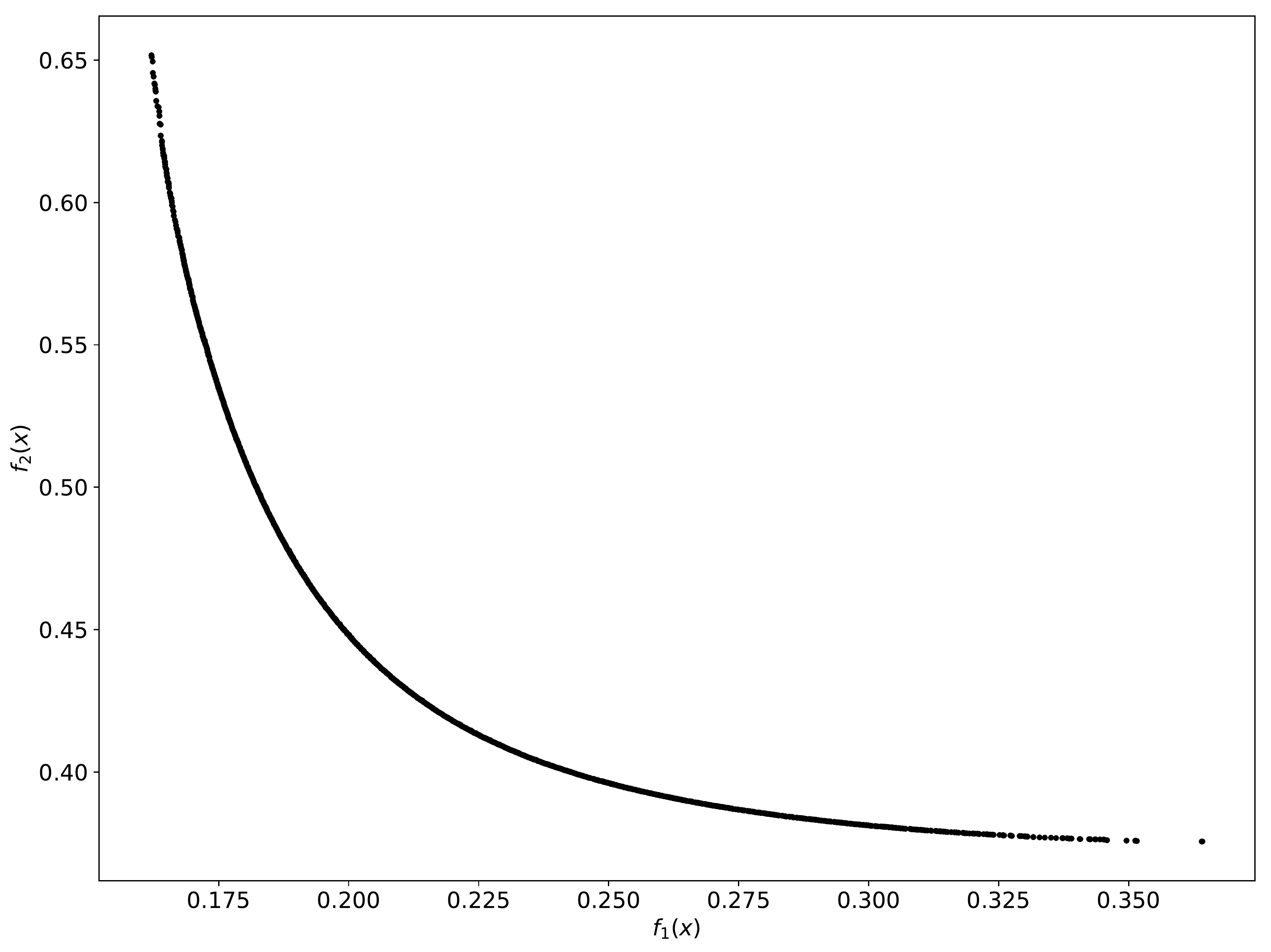}}
   \subfloat[][\textit{australian}]{\includegraphics[width=.48\textwidth, height = 4.8 cm]{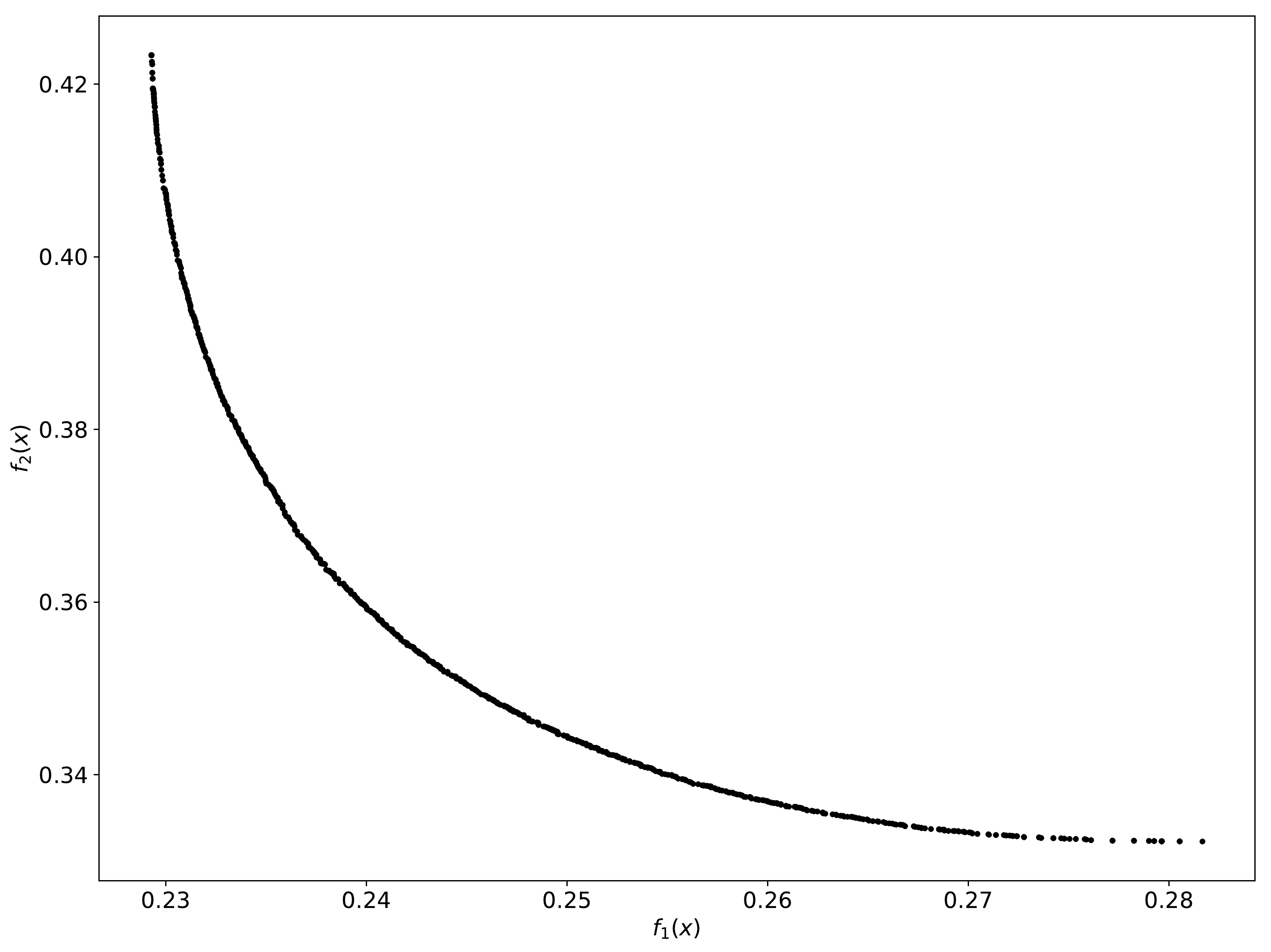}} \\
   \subfloat[][\textit{svmguide3}]{\includegraphics[width=.48\textwidth, height = 4.8 cm]{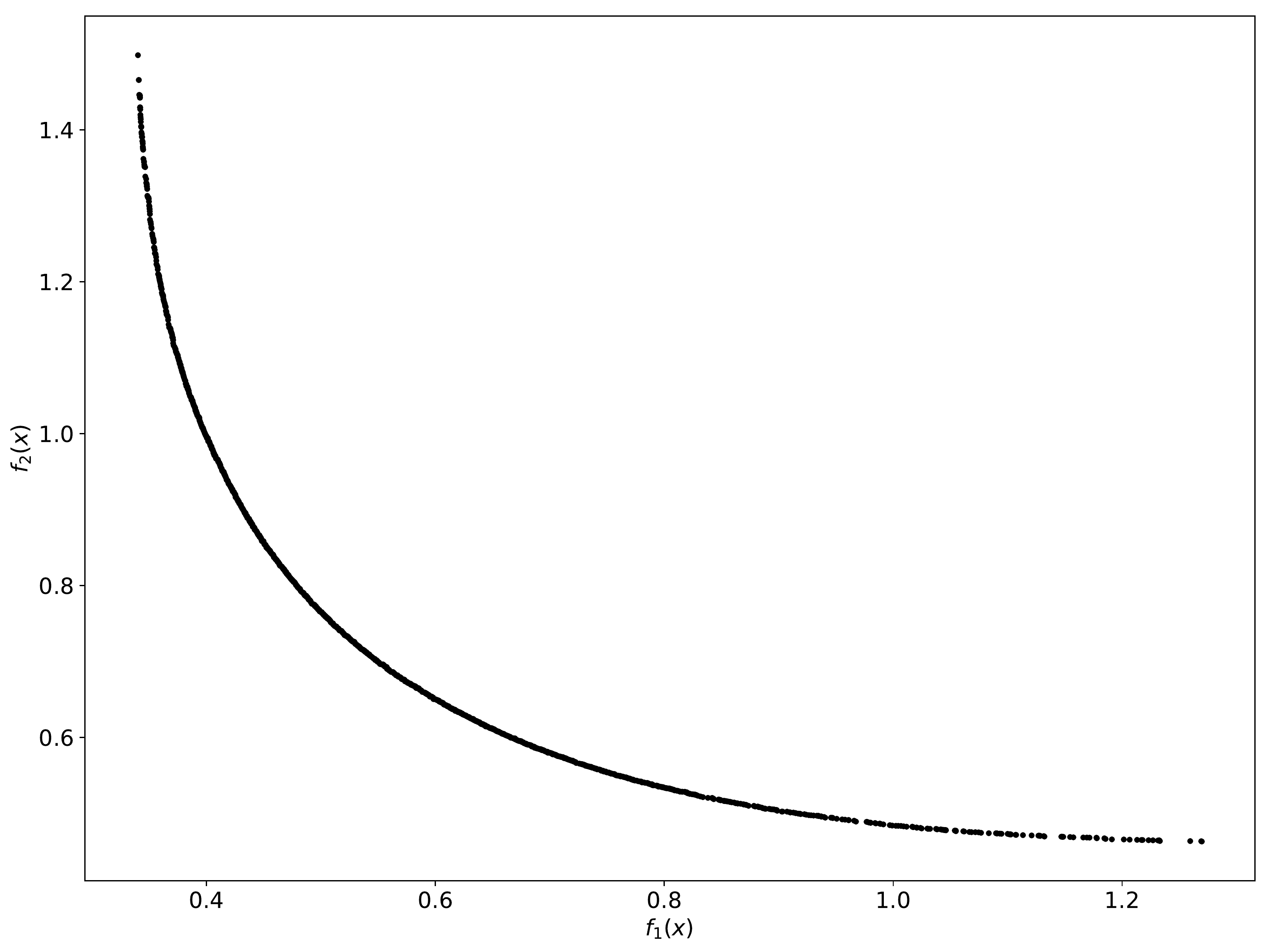}}
   \subfloat[][\textit{german.numer}]{\includegraphics[width=.48\textwidth, height = 4.8 cm]{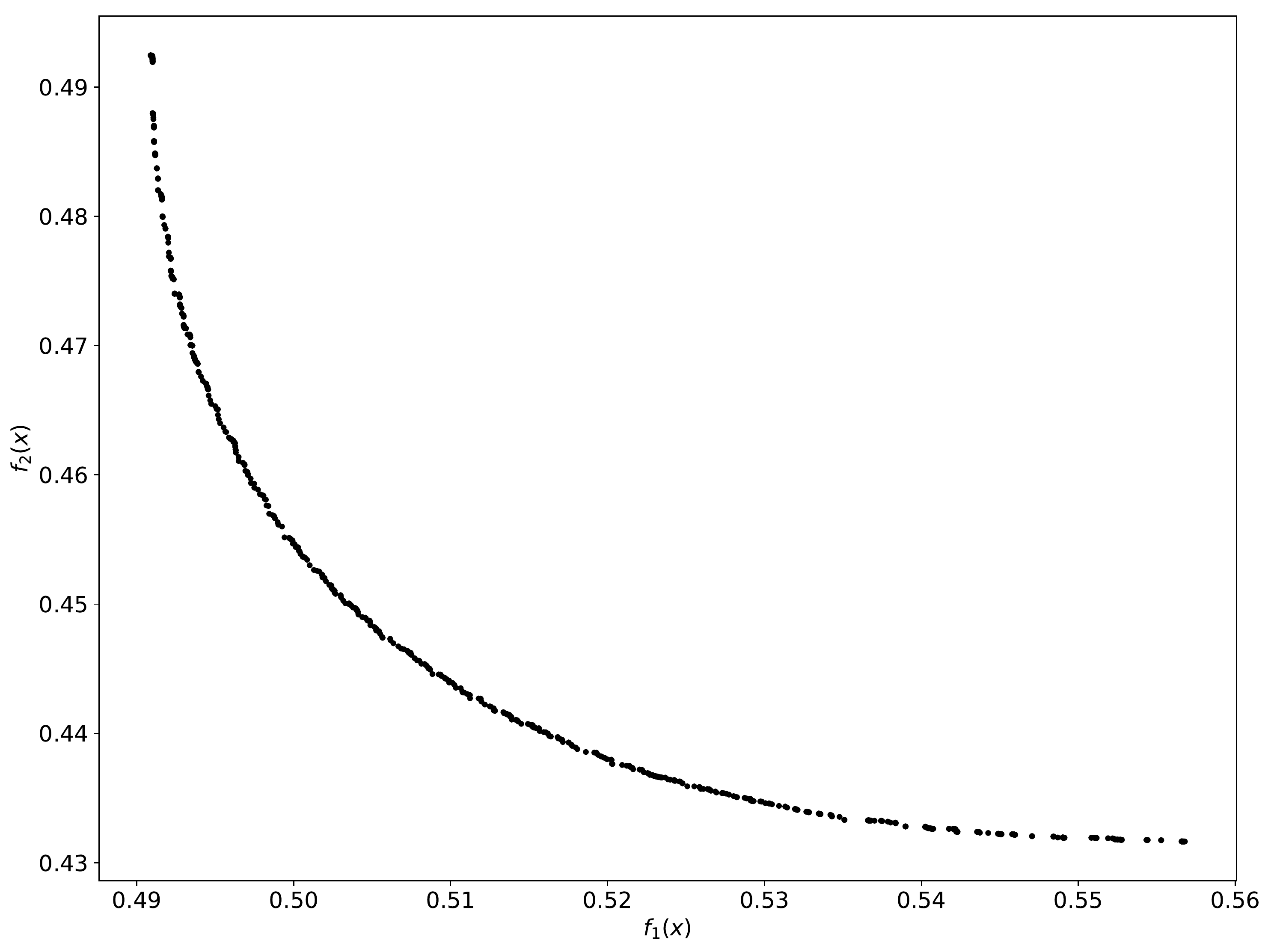}}
   \caption{The approximated Pareto fronts for the logistic regression problems: (a) \textit{heart}; (b) \textit{australian}; (c) \textit{svmguide3}; (d) \textit{german.numer.}\label{Pareto_fronts_LogRe}}
\end{figure}

The wider Pareto fronts of data sets \textit{heart} and \textit{svmguide3} coherently indicate higher distinction between their two groups. Table~\ref{PFSMG-LogRe-output-Table} presents the number of iterations and the size of Pareto front solutions when the PF-SMG algorithm is terminated. To illustrate the trade-offs, five representative points are selected from the obtained Pareto front, and the corresponding training accuracy are evaluated for the two groups separately. The CPU times for computing these Pareto fronts using PF-SMG are also included in the table (all the experiments were ran on a MacBook Pro 2019 with 2.4 GHz Quad-Core Intel Core i5). For comparison, we ran the PF-MG algorithm using full batch gradients at each iteration. Since the computation of full batch gradients is more expensive, the CPU times of PF-MG to obtain approximately the same number of non-dominated points are $0.94$, $2.5$, $1.3$, and $2.5$ times those of PF-SMG, respectively. For the dataset \textit{heart}, PF-MG takes slightly less time due to a smaller data size (using more accurate gradient results in higher convergence speed and more accurate Pareto fronts). A more systematic comparison between PF-MG and PF-SMG in terms of spread and purity is given in Section~\ref{synthetic_test}.

\begin{table}[H]
\centering
\begin{tabular}{l|c|c|c|c|c|c|c|c|c}
\toprule
          Data & \#Iter & $|\mathcal{L}_k|$ & $N$ & $P_1$ & $P_2$        & $ P_3$   & $P_4$  & $P_5$ & CPU Time \\ \midrule
\multirow{2}{*}{\textit{heart}} &\multirow{2}{*}{539}  & \multirow{2}{*}{1516} & 183 & 79.2\% & 81.4\%  & 82.0\%  & 83.1\%  & 83.6\% & \multirow{2}{*}{69.5}   \\ 
& & & 87 & 94.3\% & 90.8\% & 89.7\% & 86.2\% & 85.1\% &  \\ \hline

\multirow{2}{*}{\textit{australian}} & \multirow{2}{*}{1000}  & \multirow{2}{*}{1056}  & 468 & 84.8\% & 85.5\% & 86.1\% & 86.6\% & 86.8\% & \multirow{2}{*}{293.5} \\
& & & 222 & 91.9\% & 91.4\% & 91.0\% & 90.5\%  & 89.6\% & \\
\hline

\multirow{2}{*}{\textit{svmguide3}} & \multirow{2}{*}{161} & \multirow{2}{*}{1575} & 1182 & 80.6\% & 75.5\% & 64.5\% & 55.4\% & 32.9\% & \multirow{2}{*}{100.9} \\
& & & 61 & 29.5\% & 41.0\% & 55.7\% & 72.1\% & 85.2\% &
\\ \hline

\multirow{2}{*}{\textit{german.numer}} & \multirow{2}{*}{566} &
\multirow{2}{*}{1516}& 630 & 76.6\% & 77.1\% & 76.2\% & 75.1\% & 73.8\%  & \multirow{2}{*}{269.2} \\
& & & 370 & 78.4\% & 77.8\% & 79.0\% & 80.2\% & 80.6\% &  \\
\bottomrule
\end{tabular}
\caption{Classification accuracy corresponding to several Pareto minimizers.\label{PFSMG-LogRe-output-Table}}
\end{table}





It is observed for the groups of data sets \textit{heart} and \textit{svmguide3} that the differences of training accuracy vary more than $10\%$ among Pareto minimizers. Two important implications from the results are: (1) Given several groups of data instances for the same problem, one can evaluate their biases by observing the range of an approximated Pareto front;  (2) The resulting well-approximated Pareto front provides the decision-maker a complete trade-off of prediction accuracy across different groups. Recall that by the definition of overall accuracy equality, the larger the accuracy disparity among different groups, the higher is the bias or unfairness. To achieve a certain level of fairness in predicting new data instances, one may select a nondominated solution with a certain amount of accuracy disparity.

\subsection{Synthetic MOO test problems}
\label{synthetic_test}
\subsubsection{Test problems}

There exist more than a hundred of deterministic MOO problems reported in the literature (involving simple bound constraints), and they were collected in~\cite{ALCustodio_etal_2011}. Our testing MOO problems are taken from this collection (see~\cite{ALCustodio_etal_2011} for the problem sources) and include four cases: convex, concave, mixed (neither convex nor concave), and disconnected Pareto fronts. Table~\ref{MOO_test_prob1} provides relevant information including number of objectives, variable dimensions, simple bounds, and geometry types of Pareto fronts for the~13 selected MOO problems.

\begin{table}[H]
    \centering
    \begin{tabular}{c|c|c|c|c}
    \toprule
      Problem  & $n$ & $m$ & Simple bounds & Geometry \\ \hline
    ZDT1 & 30 & 2 & $[0, 1]$ & convex\\ \hline
     ZDT2 & 30 & 2 & $[0, 1]$ & concave\\\hline
      ZDT3  & 30 & 2 & $[0, 1]$ & disconnected\\\hline
      JOS2 & 10 & 2 & $[0, 1]$ & mixed \\ \hline
      SP1 & 2 & 2 & No & convex \\ \hline
     IM1 & 2 & 2 & $[1, 4]$
    & concave \\\hline
     FF1 & 2 & 2 & No & concave \\ \hline
     Far1 & 2 & 2 & $[-1, 1]$ & mixed \\\hline
     SK1 & 1 & 2 & No & disconnected \\ \hline
     MOP1 & 1 & 2 & No & convex \\\hline
     MOP2 & 15 & 2 & [-4, 4] & concave \\\hline
     MOP3 & 2 & 2 & $[-\pi, \pi]$ & disconnected \\ \hline
     DEB41 & 2 & 2 & $[0, 1]$ & convex \\
    \bottomrule
    \end{tabular}
    \caption{13 MOO testing problems.}
    \label{MOO_test_prob1}
\end{table}

The way to construct a corresponding stochastic MOO problem from its deterministic MOO problem was the same as in~\cite{MQuentin_PFabrice_JADesideri_2018}. For each of these MOO test problems, we added random noise to its variables to obtain a stochastic MOO problem, i.e.,
\begin{equation*}
\begin{array}{rl}
     \min & F(x) = (\mathbb{E}[f_1(x + w)], \ldots, \mathbb{E}[f_m(x + w)])^\top \\ [0.5ex]
     \mbox{s.t.} & x \in \mathcal{X},
\end{array}
\end{equation*}
where $w$ is uniformly distributed with mean zero and interval length being $1/10$ of the length of the simple bound interval (the latter one was artificially chosen when not given in the problem description).
Note that the stochastic gradients will not be unbiased estimates of the true gradients of each objective function, but rather gradients of randomly perturbed points in the neighborhood of the current point.

Figure~\ref{geometry_of_Pareto_fronts} illustrates four different geometry shapes of Pareto fronts obtained by removing all dominated points from the union of the resulting Pareto fronts obtained from the application of the PF-SMG and PF-MG algorithms. In the next subsection, the quality of approximated Pareto fronts obtained from the two algorithms is measured and compared in terms of the Purity and Spread metrics.

\begin{figure}[bt]
   \centering
   \subfloat[][SP1]{\includegraphics[width=.45\textwidth]{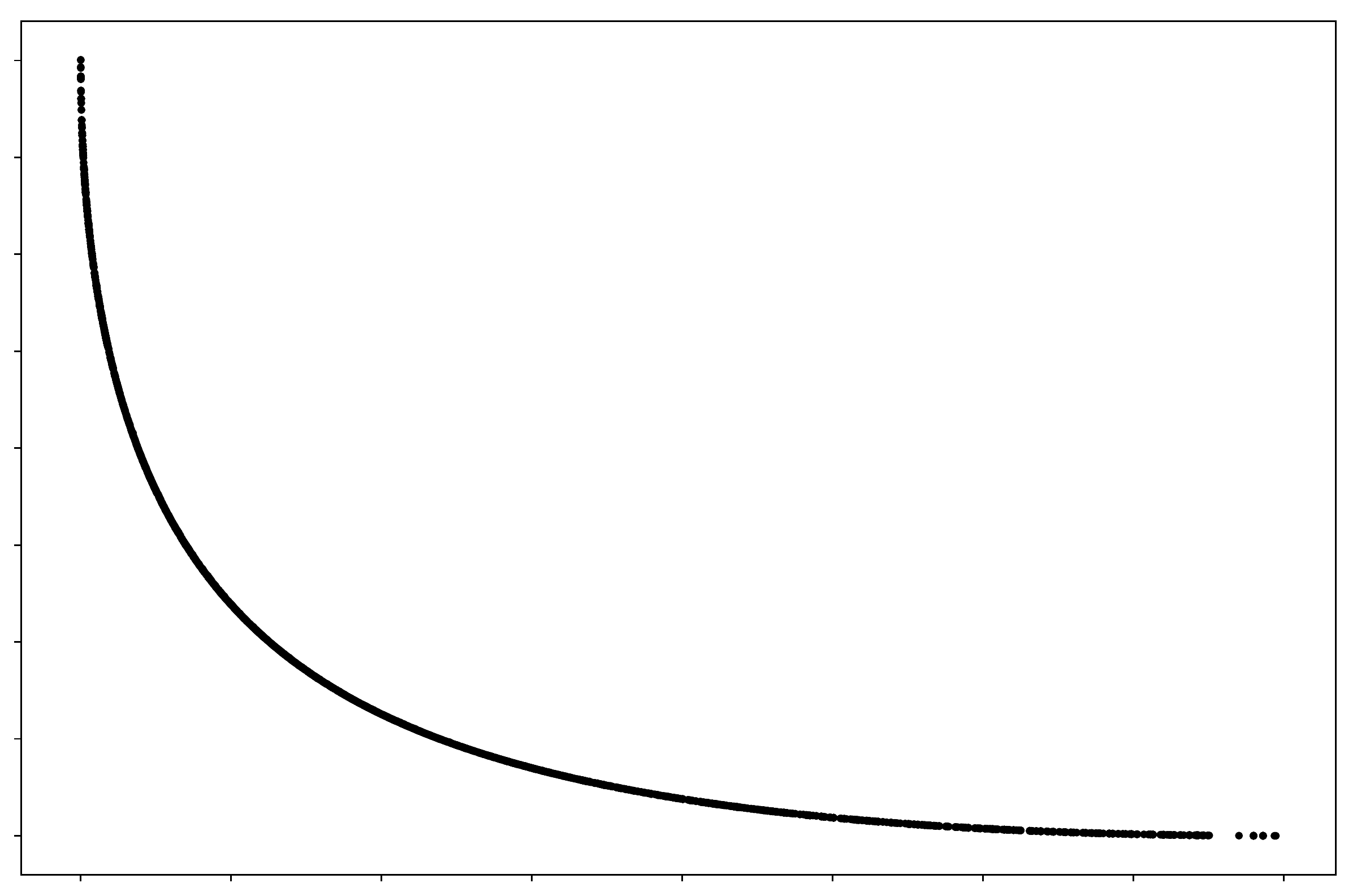}}\quad
   \subfloat[][FF1]{\includegraphics[width=.45\textwidth]{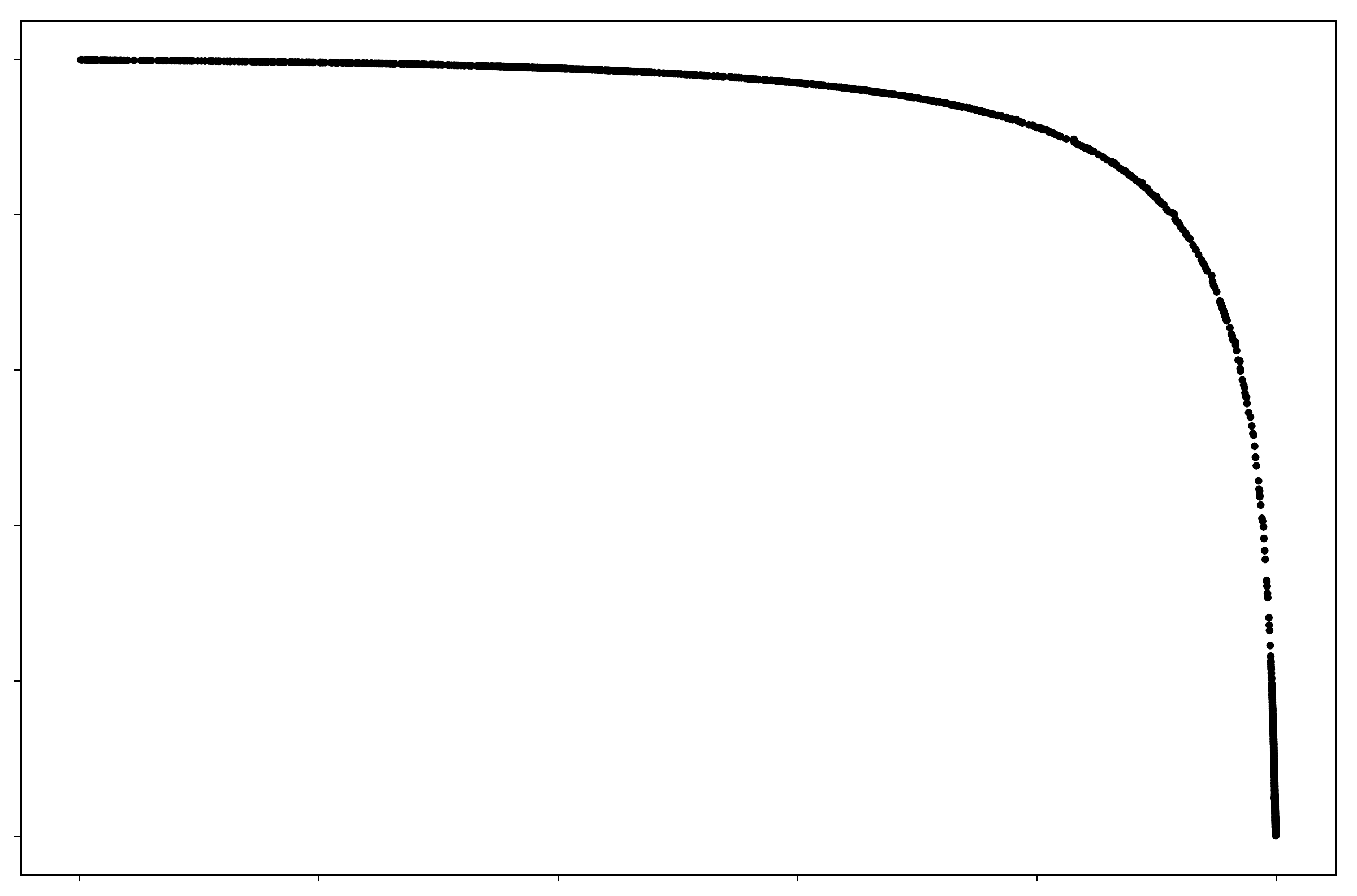}}\\
   \subfloat[][JOS2]{\includegraphics[width=.45\textwidth]{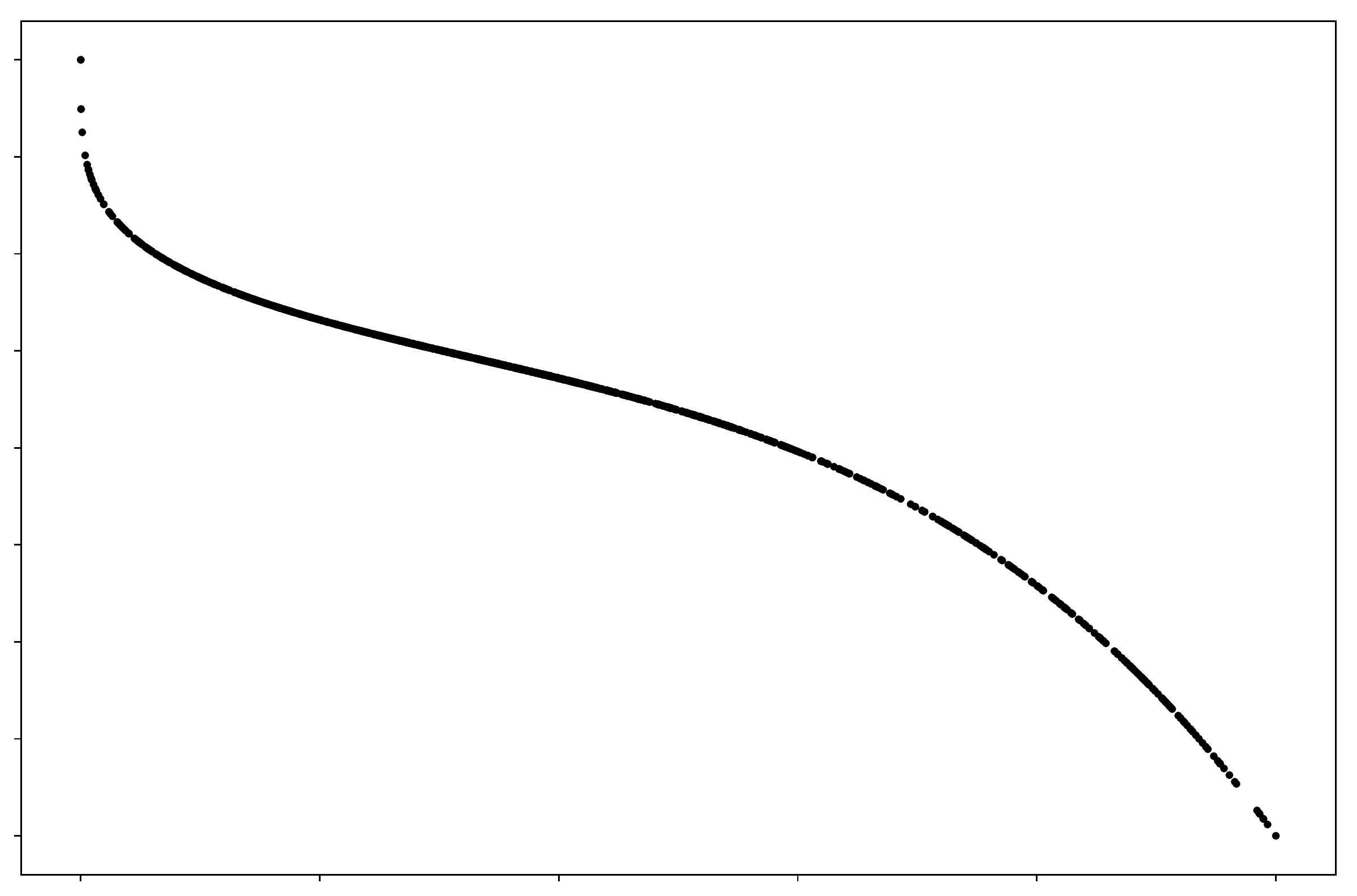}}\quad
   \subfloat[][ZDT3]{\includegraphics[width=.45\textwidth]{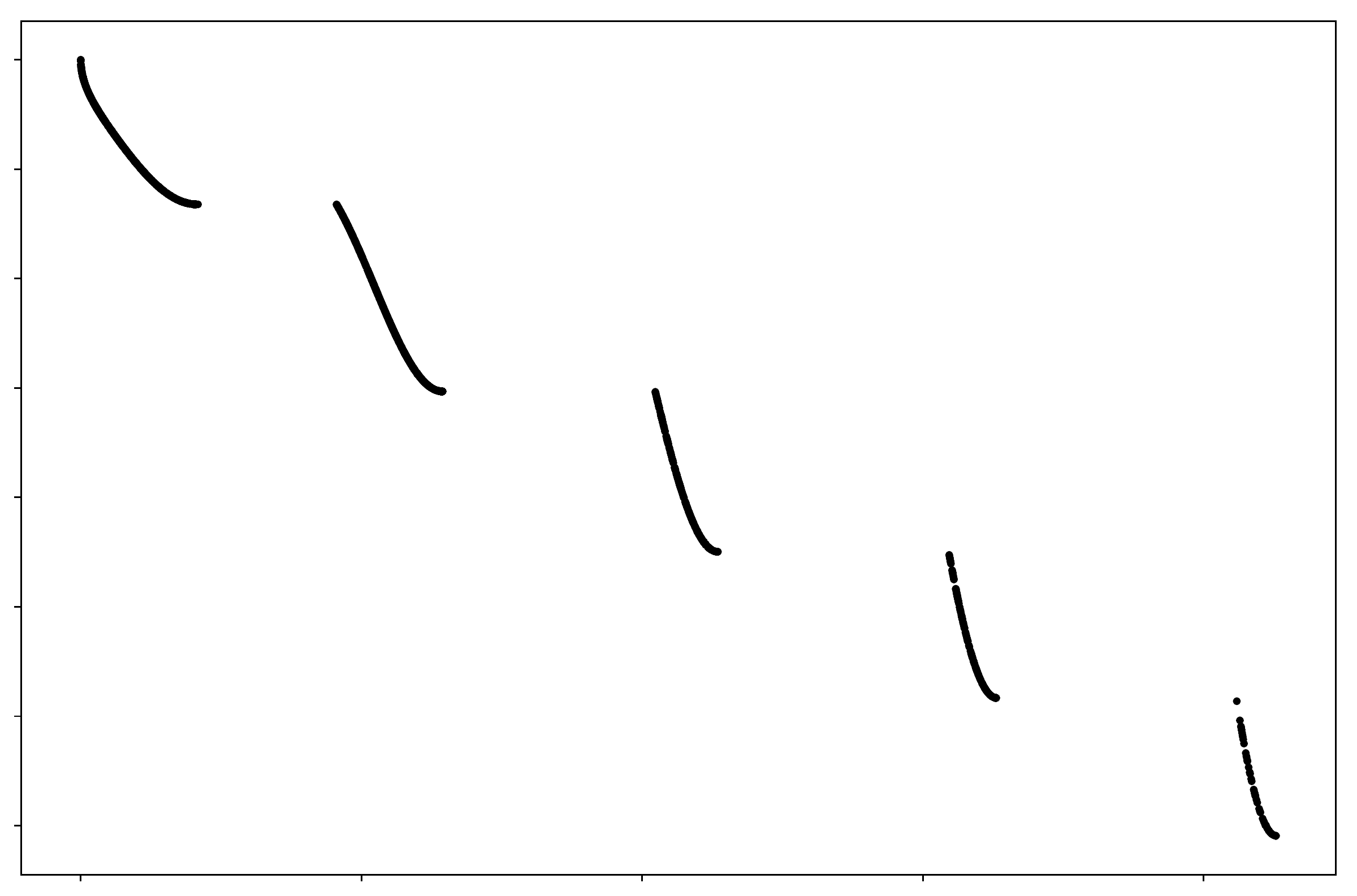}}
   \caption{Different geometry shapes of Pareto fronts: (a) Convex; (b) Concave; (c) Mixed (neither convex nor concave); (d) Disconnected.}
   \label{geometry_of_Pareto_fronts}
\end{figure}

\subsubsection{Numerical results and analysis}

For all problems, the initial step size was set to 0.3 for both PF-SMG and PF-MG algorithms. To be fair, we ran 10 times the PF-SMG algorithm for each problem and selected the one with the average value in $\Gamma$, i.e., the maximum size of the holes. (Although there is some randomization in PF-MG, its output does not differ significantly from one run to another.) The quality of the obtained Pareto fronts, the number of iterations, and the size of the Pareto front approximations when the algorithms are terminated are reported in Table~\ref{MOO_test_prob_metrics}. We also plot performance profiles, see Figure~\ref{preformance_profiles}, in terms of Purity and the two formula of Spread metrics.

\begin{figure}[H]
   \centering
   \subfloat[][Purity]{\includegraphics[width=.34\textwidth, angle=-90]{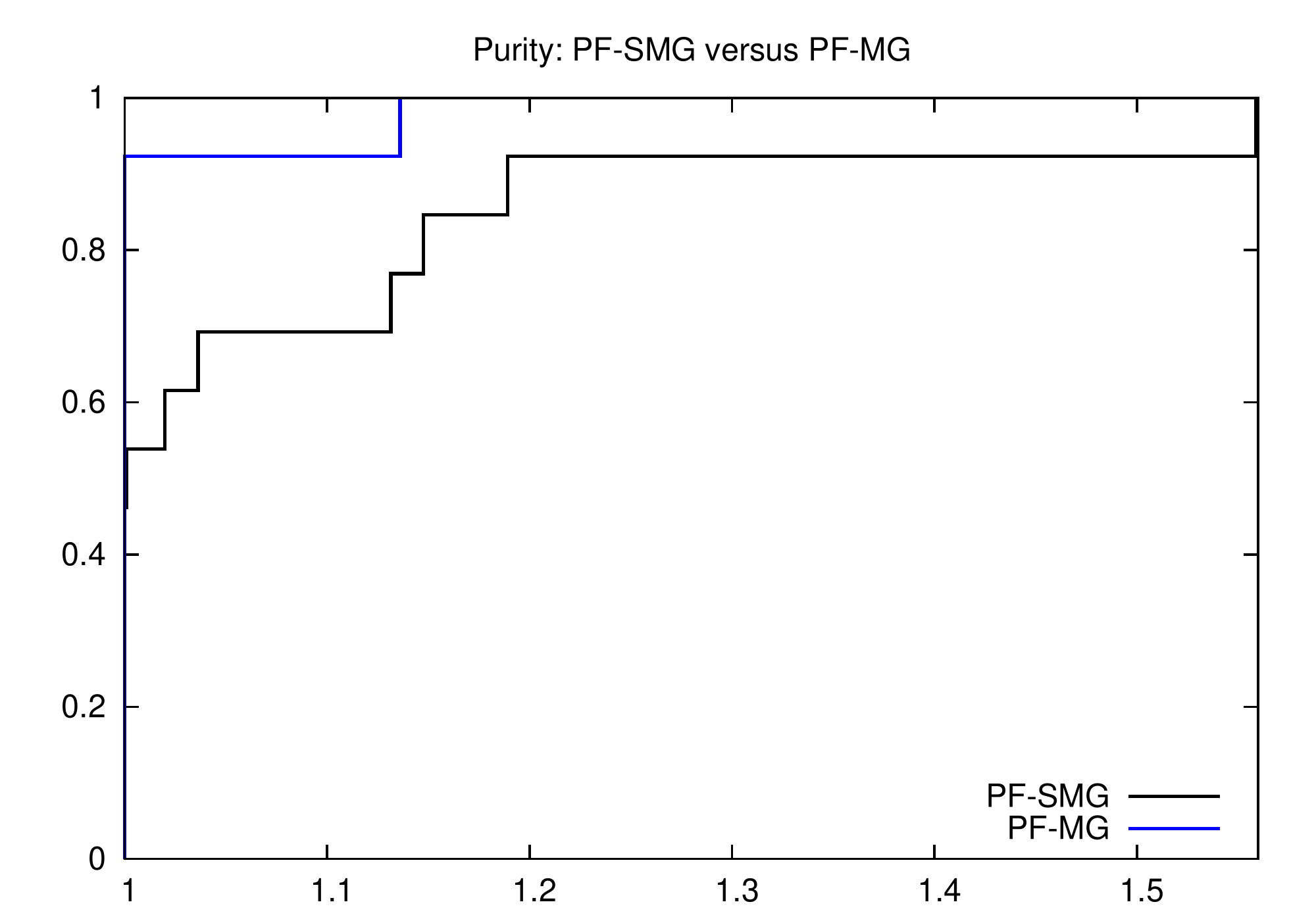}}\\
   \subfloat[][Spread: $\Gamma$]{\includegraphics[width=.34\textwidth, angle=-90]{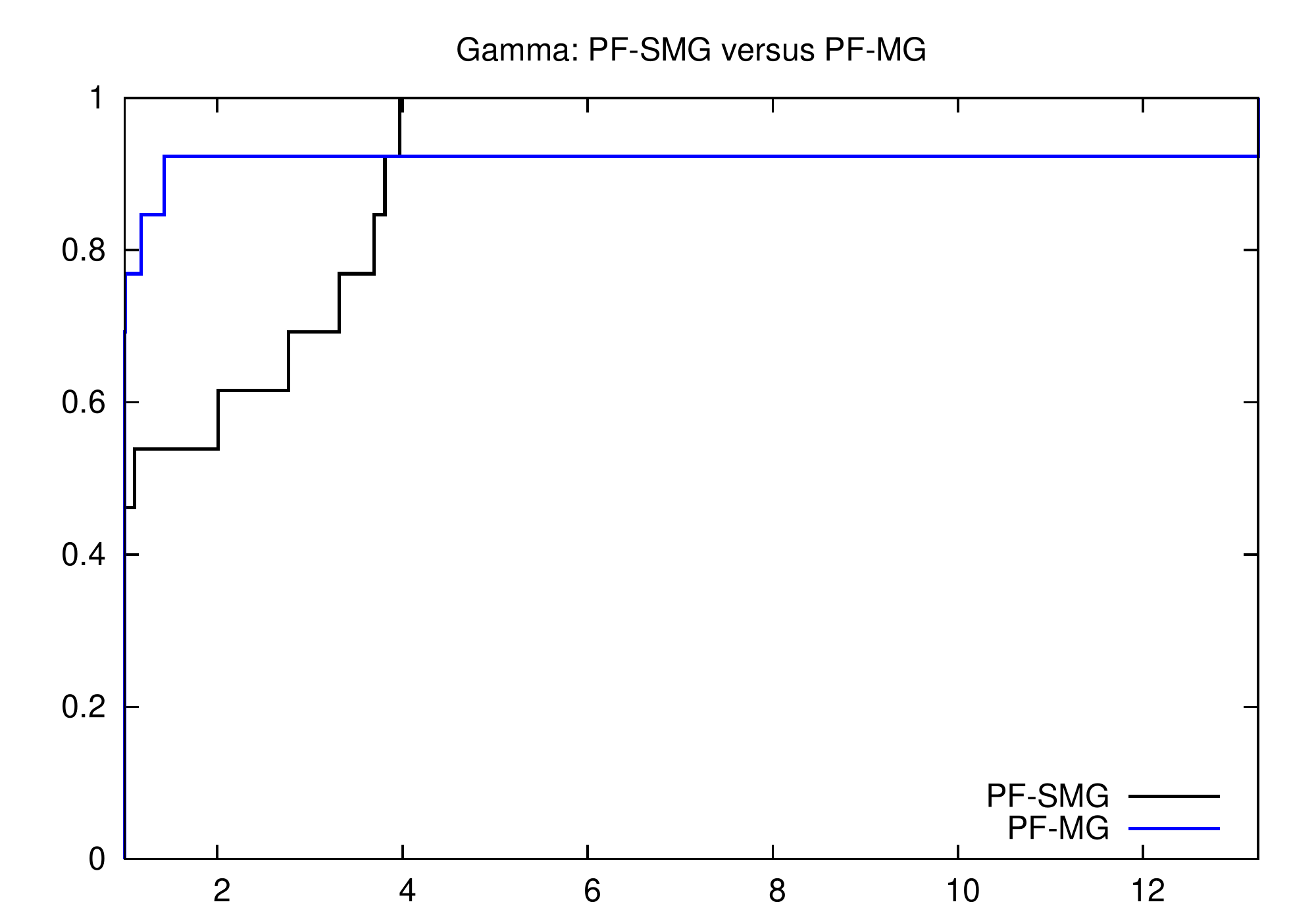}}\quad
   \subfloat[][Spread: $\Delta$]{\includegraphics[width=.34\textwidth, angle=-90]{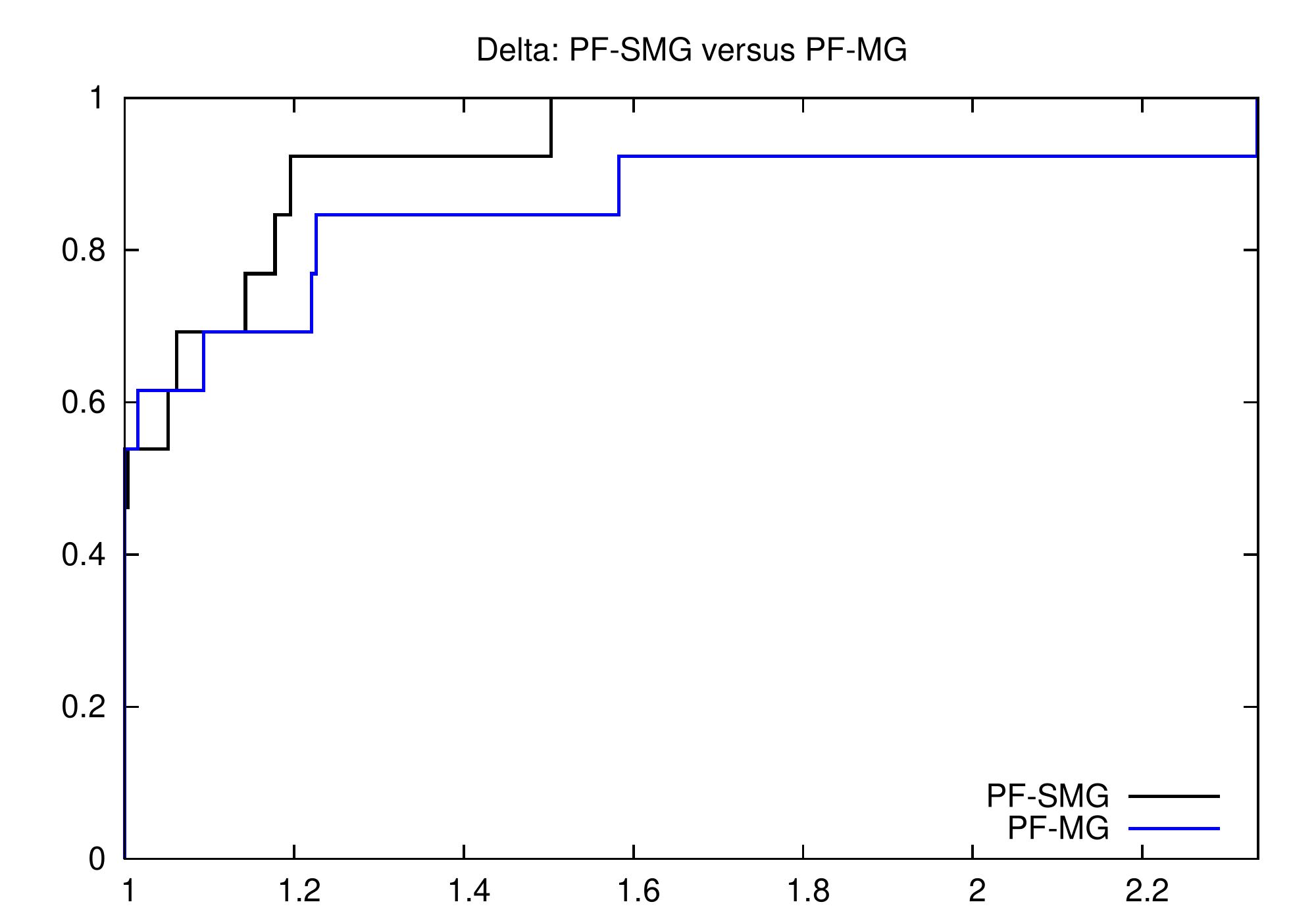}}
   \caption{Performance profiles in terms of Purity, $\Gamma$, and $\Delta$ repectively.}
   \label{preformance_profiles}
\end{figure}

Overall, the PF-MG algorithm produces Pareto fronts of higher Purity than the PF-SMG, which is reasonable since using the accurate gradient information results in points closer to the true Pareto front.
However, the Purity of Pareto fronts resulting from the PF-SMG is quite close to the one from the PF-MG in most of the testing problems.
Also, when we examine the quality of the fronts in terms of the Spread metrics (see $\Gamma$ and $\Delta$ in Table~\ref{MOO_test_prob_metrics}), their performances are comparable, which indicates that the proposed PF-SMG algorithm is able to produce well-spread Pareto fronts. For some problems like IM1 and FF1, it is observed that PF-SMG generates nondominated points faster than the PF-MG. This might be due to the fact PF-SMG has two sources of stochasticity, both in generating the points and in applying stochastic multi-gradient, whereas PF-MG is only stochastic in the generation of points.

On the other hand, perhaps due to the worse accuracy of stochastic multi-gradients, PF-SMG takes more iterations than PF-MG to achieve the same tolerance level. Nevertheless, suppose that the computational cost for computing the true gradients for each objective function is significantly higher than the one for obtaining the stochastic gradients. It is easy to consider scenarios when the computation cost of PF-MG would be far more expensive than for PF-SMG.

\clearpage
\begin{table}[H]
    \centering
    \begin{tabular}{c|c|cccccc}
    \toprule
     Problem & Algorithm & Purity & $\Gamma$ & $\Delta$ & \# Iter & $|\mathcal{L}_k|$\\ \midrule
     \multirow{2}{*}{ZDT1}
     & PF-MG & 1.000 & 0.0332 & 1.4404 & 26 & 1575 \\
     & PF-SMG & 1.000 & 0.0666 & 1.6958 & 26 & 1789 \\ \hline
     \multirow{2}{*}{ZDT2}
     & PF-MG & 1.000 & 0.9336 & 1.0407 & 48 & 1524 \\
     & PF-SMG & 1.000 & 0.0705 & 1.5637 & 32 & 1680 \\ \hline
     \multirow{2}{*}{ZDT3}
     & PF-MG & 0.999 & 0.1716 & 1.5941 & 84 & 1524 \\
     & PF-SMG & 0.999 & 0.6539 & 1.3005 & 70 & 1544\\ \hline
     \multirow{2}{*}{JOS2}
     & PF-MG & 1.000 & 0.1853 & 1.3520 & 24 & 1530 \\
     & PF-SMG & 1.000 & 0.7358 & 1.5445 & 18 & 2271 \\ \hline
     \multirow{2}{*}{SP1}
     & PF-MG & 0.996 & 0.0763 & 1.5419 & 24 & 1826 \\
     & PF-SMG & 0.880 & 0.2817 & 0.9742 & 102 & 1503 \\ \hline
     \multirow{2}{*}{IM1}
     & PF-MG & 0.992 & 0.0936 & 0.8879 & 18 & 1581 \\
     & PF-SMG & 0.973 & 0.2591 & 1.0613 & 16 & 2161 \\ \hline
     \multirow{2}{*}{FF1}
     & PF-MG & 0.982 & 0.0788 & 1.5637 & 46 & 1533\\
     & PF-SMG  & 0.630 & 0.0671 & 1.5701 & 20 & 1834\\ \hline
     \multirow{2}{*}{Far1}
     & PF-MG & 0.843 & 0.3800 & 1.5072 & 26 & 1741\\
     & PF-SMG & 0.958 & 0.4192 & 1.5996 & 44 & 1602 \\ \hline
     \multirow{2}{*}{SK1}
     & PF-MG & 1.000 & 24.6399 & 1.0053 & 68 & 1531 \\
     & PF-SMG & 0.999 & 24.6196 & 0.9195 & 48 & 1614 \\ \hline
     \multirow{2}{*}{MOP1}
     & PF-MG & 1.000 & 0.0329 & 0.9003 & 78 & 1505 \\
     & PF-SMG & 1.000 & 0.1091 & 0.9462 & 14 & 2036 \\ \hline
     \multirow{2}{*}{MOP2}
     & PF-MG & 1.000 & 0.0614 & 1.8819 & 140 & 1527 \\
     & PF-SMG & 0.841 & 0.0609 & 0.8057 & 124 & 1504 \\ \hline
     \multirow{2}{*}{MOP3}
     & PF-MG & 0.990 & 19.8772 & 1.7938 & 26 & 1530 \\
     & PF-SMG & 0.863 & 19.8667 & 1.7664 & 50 & 1571 \\ \hline
     \multirow{2}{*}{DEB41}
     & PF-MG & 0.953 & 26.8489 & 1.8430 & 14 & 1813\\
     & PF-SMG & 0.920 & 18.8147 & 1.5101 & 18 & 1997 \\
    \bottomrule
    \end{tabular}
    \caption{Comparison between resulting Pareto fronts from the PF-MG and PF-SMG algorithms.\label{MOO_test_prob_metrics}}
\end{table}

Two informative final notes. The Pareto fronts of problem SK1 and MOP3 are disconnected, and hence, their values of $\Gamma$ are  significantly larger than others. There exists a conflict between depth (Purity) and breadth (Spread) of the Pareto front. One can always tune some parameters, e.g., the number of starting points and the number of points generated per point at each iteration, to balance the Purity and Spread of the resulting Pareto fronts.

\section{Conclusions}
\label{conclude}

The stochastic multi-gradient (SMG) method is an extension of the stochastic gradient method from single to multi-objective optimization (MOO). However, even based on the assumption of unbiasedness of the stochastic gradients of the individual functions, it has been observed
in this paper that there exists a bias between the stochastic multi-gradient and the corresponding true multi-gradient, essentially due to the composition with the solution of a quadratic program (see~\eqref{subproblem3}). Imposing a condition on the amount of tolerated biasedness, we established sublinear convergence rates, $\mathcal{O}(1/k)$ for strongly convex and $\mathcal{O}(1/\sqrt{k})$ for convex objective functions, similar to what is known for single-objective optimization, except that the optimality gap was measured in terms of a weighted sum of the individual functions. We realized that the main difficulty in establishing these rates for the multi-gradient method came from the unknown limiting behavior of the weights generated by the algorithm. Nonetheless, our theoretical results contribute to a deeper understanding of the convergence rate theory of the classical stochastic gradient method in the MOO setting.

To generate an approximation of the entire Pareto front in a single run, the SMG algorithm was framed into a Pareto-front one, iteratively updating a list of nondominated points. The resulting PF-SMG algorithm was shown to be a robust technique for smooth stochastic MOO since it has produced well-spread and sufficiently accurate Pareto fronts, while being relatively efficient in terms of the overall computational cost.
Our numerical experiments on binary logistic regression problems showed that solving a well-formulated MOO problem can be a novel tool for identifying biases among potentially different sources of data and improving the prediction fairness.

As it is well known, noise reduction~\cite{ADefazio_FBach_SLacoste-Julien_2014,RJohnson_TZhang_2013,BTPolyak_ABJuditsky_1992,SShalev-Shwartz_etal_2011} was studied intensively during the last decade to improve the performance of the
stochastic gradient method. Hence, a relevant topic for our future research is the
study of noise reduction in the setting of the stochastic multi-gradient method for MOO.
More applications and variants of the algorithm can be further explored.
For example, we have not yet tried to solve stochastic MOO problems when the feasible region is different from box constraints.
We could also consider the incorporation of a proximal term and in doing so we could handle nonsmooth regularizers.
Other models arising in supervised machine learning, such as the deep learning, could be also framed into an MOO context.
Given that the neural networks used in deep learning give rise to nonconvex objective functions, we would also be interested in developing the convergence rate theory for the SMG algorithm in the nonconvex case.

\appendix
\section{Proof of Theorem~\ref{th2:convex}}
\label{append_proof_conv}
\begin{proof}
By applying inequalities \eqref{ineq:theorem_bound1}, \eqref{ineq:theorem_bound2}, and \eqref{subprob_lip_lambda2} to \eqref{proof_ineq_step1}, one obtains
\begin{equation}
\label{proof_theorem_step1}
\begin{split}
\mathbb{E}_{w_k}[ \| x_{k+1} - x_*\|^2 ] \;\leq\; & \| x_{k} - x_*\|^2  + \alpha_k^2(L^2_g + 2\Theta(L_{\nabla S}+ \beta L_{\nabla S}M_{S}))  \\ & 2\alpha_k \mathbb{E}_{w_k}[\nabla_x S(x_k, \lambda_k)]^\top (x_k - x_*).
\end{split}
\end{equation}
From convexity one has
\begin{equation}
    S(x_k, \lambda_k) - S(x_*, \lambda_k)  \;\leq\;  \nabla_x S(x_k,\lambda_k)^\top(x_k - x_*).
    \label{ineq:convex}
\end{equation}
Then, plugging \eqref{ineq:convex} into inequality \eqref{proof_theorem_step1} and rearranging yield
\begin{align*}
    2\alpha_k (\mathbb{E}_{w_k}[S(x_k, \lambda_k)] -  \mathbb{E}_{w_k}[S(x_*, \lambda_k)])  \;\leq\; & \| x_{k} - x_*\|^2  - \mathbb{E}_{w_k}[\| x_{k+1} - x_*\|^2] \\
    & + \alpha_k^2(L^2_g + 2\Theta(L_{\nabla S}+ \beta L_{\nabla S}M_{S})).
\end{align*}
For simplicity denote $\hat{M} = L^2_g + 2\Theta(L_{\nabla S}+ \beta L_{\nabla S}M_{S})$. Dividing both sides by $\alpha_k$ and taking total expectations on both sides allow us to write
\begin{equation*}
    \begin{split}
   2(\mathbb{E}[S(x_k, \lambda_k)] -  \mathbb{E}[S(x_*, \lambda_k)]) \;\leq\;  & \frac{\mathbb{E}[\| x_{k} - x_*\|^2] - \mathbb{E}[\| x_{k+1} - x_*\|^2]}{\alpha_k} + \alpha_k \hat{M}.
    \end{split}
\end{equation*}
Replacing $k$ by $s$ in the above inequality and summing over $s = 1,\ldots, k$ lead to
\begin{align*}
    2\sum_{s= 1}^k (\mathbb{E}[S(x_s, \lambda_s)] -  \mathbb{E}[S(x_*, \lambda_s)])
     \;\leq\;  & \frac{1}{\alpha_1}\mathbb{E}[\|x_{1} - x_*\|^2] + \sum_{s= 1}^k \alpha_s\hat{M}\\
     & + \sum_{s = 2}^k (\frac{1}{\alpha_{s}} - \frac{1}{\alpha_{s - 1}})\mathbb{E}[\| x_{s} - x_*\|^2] \\
     \;\leq\; & \frac{\Theta^2}{\alpha_1} + \sum_{s = 2}^k (\frac{1}{\alpha_{s}} - \frac{1}{\alpha_{s - 1}}) \Theta^2 + \sum_{s = 1}^k \alpha_s \hat{M} \\
     \;\leq\; & \frac{\Theta^2}{\alpha_k} + \sum_{s = 1}^k \alpha_s \hat{M}.
\end{align*}
Then, using $\alpha_s = \frac{\bar{\alpha}}{\sqrt{s}}$ and dividing both sides by $2k$ in the last inequality give us
\begin{equation}
\label{theorem2::last_step2}
    \frac{1}{k}\sum_{s= 1}^k (\mathbb{E}[S(x_s, \lambda_s)] -  \mathbb{E}[S(x_*, \lambda_s)]) \;\leq \;\frac{\Theta^2}{2\bar{\alpha}\sqrt{k}}  + \frac{\bar{\alpha}\hat{M}}{\sqrt{k}},
\end{equation}
from the fact $\sum_{s = 1}^k \frac{\bar{\alpha}}{\sqrt{s}} \leq 2\bar{\alpha} \sqrt{k}$.  In the left-hand side, one can use the following inequality
    \begin{equation}
    \label{theorem2::last_step1}
      \min_{s = 1, \ldots, k} \mathbb{E}[S({{x}_k, \lambda}_k)] - \mathbb{E}[S(x_*, \bar{\lambda}_k)] \;\leq\; \frac{1}{k}\sum_{s= 1}^k (\mathbb{E}[S(x_s, \lambda_s)] -  \mathbb{E}[S(x_*, \lambda_s)]),
   \end{equation}
where $\bar{\lambda}_k = \textstyle \frac{1}{k} \sum_{s = 1}^k \lambda_s$.
The final result follows from combining~\eqref{theorem2::last_step2} and~\eqref{theorem2::last_step1}.
\end{proof}

\section{Metrics for comparison}
\label{append_metrics_comp}
Let $\mathcal{A}$ denote the set of algorithms/solvers and $\mathcal{T}$ denote the set of test problems. The Purity metric measures the accuracy of an approximated Pareto front. Let us denote $H(\mathcal{P}_{a, t})$ as an approximated Pareto front of problem $t$ computed by algorithm~$a$. We approximate the ``true'' Pareto front $H(\mathcal{P}_t)$ for problem~$t$ by all the nondominated points in $\cup_{a \in \mathcal{A}} H(\mathcal{P}_{a, t})$. Then, the Purity of a Pareto front computed by algorithm $a$ for problem $t$ is the ratio $r_{a, t} = |H(\mathcal{P}_{a, t}) \cap H(\mathcal{P}_t)|/|H(\mathcal{P}_{a, t})| \in [0, 1]$, which calculates the percentage of ``true'' nondominated solutions among all the nondominated points generated by algorithm $a$. A higher ratio value corresponds to a more accurate Pareto front. In our context, it is highly possible that the Pareto front obtained from PF-MG algorithm dominates that from the PF-SMG algorithm since the former one uses true gradients.

The Spread metric is designed to measure the extent of the point spread in a computed Pareto front, which requires the computation of extreme points in the objective function space~$\mathbb{R}^m$.
Among the $m$ objective functions, we select a pair of nondominated points in $\mathcal{P}_t$ with the highest pairwise distance (measured using $h_i$) as the pair of extreme points.
More specifically, for a particular algorithm $a$, let $(x_{\min}^i, x_{\max}^i) \in \mathcal{P}_{a, t}$ denote the pair of nondominated points where $x_{\min}^i = \argmin_{x \in \mathcal{P}_{a, t}} h_i(x)$ and $x_{\max}^i = \argmax_{x \in \mathcal{P}_{a, t}} h_i(x)$. Then, the pair of extreme points is $(x_{\min}^k, x_{\max}^k)$ with $k = \argmax_{i = 1, \ldots, m} h_i(x_{\max}^i) - h_i(x_{\min}^i)$.

The first Spread formula calculates the maximum size of the holes for a Pareto front. Assume algorithm $a$ generates an approximated Pareto front with $M$ points, indexed by $1, \ldots, M$, to which the extreme points $H(x_{\min}^k)$,$H(x_{\max}^k)$ indexed by $0$ and $M+1$ are added. Denote the maximum size of the holes by $\Gamma$. We have
\begin{equation*}
\label{Spread_Gamma}
    \Gamma \;=\; \Gamma_{a, t} \;=\; \max_{i \in \{1, \ldots, m\}} \left(\max_{j \in \{1, \ldots, M\}}\{\delta_{i,j}\}\right),
\end{equation*}
where $\delta_{i,j} = h_{i,j + 1} - h_{i, j}$, and we assume each of the objective function values $h_i$ is sorted in an increasing order.

The second formula was proposed by~\cite{KDeb_etal_2002} for the case $m = 2$ (and further extended to the case $m \geq 2$ in~\cite{ALCustodio_etal_2011}) and indicates how well the points are distributed in a Pareto front. Denote the point spread by~$\Delta$. It is computed by the following formula:
\begin{equation*}
\label{Spread_Delta}
    \Delta \;=\; \Delta_{a, t} \;=\; \max_{i \in \{1, \ldots, m\}} \left(\frac{\delta_{i, 0} + \delta_{i, M} + \sum_{j = 1}^{M-1}|\delta_{i, j} - \bar{\delta}_i|}{\delta_{i, 0} + \delta_{i, M} + (M-1)\bar{\delta}_i} \right),
\end{equation*}
where $\bar{\delta}_i, i = 1, \ldots, m$ is the average of $\delta_{i, j}$ over $j = 1, \ldots, M -1$. Note that the lower $\Gamma$ and $\Delta$ are, the more well distributed the Pareto front is.

\small
\bibliographystyle{plain}
\bibliography{smg-moo}

\begin{thebibliography}{10}

\bibitem{FBAbdelaziz_1992}
F.~B. Abdelaziz.
\newblock {\em L'efficacit{\'e} en Programmation Multi-Objectifs Stochastique}.
\newblock PhD thesis, Universit{\'e} de Laval, Qu{\'e}bec, 1992.

\bibitem{FBAbdelaziz_2012}
F.~B. Abdelaziz.
\newblock Solution approaches for the multiobjective stochastic programming.
\newblock {\em European J. Oper. Res.}, 216:1--16, 2012.

\bibitem{SBandyopadhya_SKPal_BAruna_2004}
S.~Bandyopadhya, S.~K. Pal, and B.~Aruna.
\newblock Multiobjective {GA}s, quantitative indices, and pattern
  classification.
\newblock {\em IEEE Transactions on Systems, Man, and Cybernetics, Part B
  (Cybernetics)}, 34:2088--2099, 2004.

\bibitem{SBandyopadhyay_etal_2008}
S.~Bandyopadhyay, S.~Saha, U.~Maulik, and K.~Deb.
\newblock {A simulated annealing-based multiobjective optimization algorithm:
  AMOSA}.
\newblock {\em IEEE Transactions on Evolutionary Computation}, 12:269--283,
  2008.

\bibitem{SBarocas_MHardt_ANarayanan_2017}
S.~Barocas, M.~Hardt, and A.~Narayanan.
\newblock Fairness in machine learning.
\newblock {\em NIPS Tutorial}, 1, 2017.

\bibitem{RBerk_etal_2018}
R.~Berk, H.~Heidari, S.~Jabbari, M.~Kearns, and A.~Roth.
\newblock Fairness in criminal justice risk assessments: {T}he state of the
  art.
\newblock {\em Sociological Methods \& Research}, pages 1--42, 2018.

\bibitem{HBonnel_ANIusem_BFSvaiter_2005}
H.~Bonnel, A.~N. Iusem, and B.~F. Svaiter.
\newblock Proximal methods in vector optimization.
\newblock {\em SIAM J. Optim.}, 15:953--970, 2005.

\bibitem{LBottou_FECurtis_JNocedal_2018}
L.~Bottou, F.~E. Curtis, and J.~Nocedal.
\newblock Optimization methods for large-scale machine learning.
\newblock {\em SIAM Rev.}, 60:223--311, 2018.

\bibitem{RCaballero_2004}
R.~Caballero, E.~Cerd\'a, M.~Munoz, and L.~Rey.
\newblock Stochastic approach versus multiobjective approach for obtaining
  efficient solutions in stochastic multiobjective programming problems.
\newblock {\em European J. Oper. Res.}, 158:633--648, 2004.

\bibitem{VChandrasekaran_etal_2012}
V.~Chandrasekaran, B.~Recht, P.~A. Parrilo, and A.~S. Willsky.
\newblock The convex geometry of linear inverse problems.
\newblock {\em Found. Comput. Math.}, 12:805--849, 2012.

\bibitem{CCChang_CJLin_2011}
C.~C. Chang and C.~J. Lin.
\newblock {LIBSVM}: A library for support vector machines.
\newblock {\em ACM Transactions on Intelligent Systems and Technology (TIST)},
  2:27, 2011.

\bibitem{KLChung_1954}
K.~L. Chung.
\newblock On a stochastic approximation method.
\newblock {\em Ann. Math. Statist.}, 25:463--483, 1954.

\bibitem{ALCustodio_etal_2011}
A.~L. Cust{\'o}dio, J.~A. Madeira, A.~I.~F. Vaz, and L.~N. Vicente.
\newblock Direct multisearch for multiobjective optimization.
\newblock {\em SIAM J. Optim.}, 21:1109--1140, 2011.

\bibitem{IDas_JEDennis_1998}
I.~Das and J.~E. Dennis.
\newblock Normal-boundary intersection: A new method for generating the
  {P}areto surface in nonlinear multicriteria optimization problems.
\newblock {\em SIAM J. Optim.}, 8:631--657, 1998.

\bibitem{KDeb_etal_2002}
K.~Deb, A.~Pratap, S.~Agarwal, and T.~Meyarivan.
\newblock {A fast and elitist multiobjective genetic algorithm: NSGA-II}.
\newblock {\em IEEE Transactions on Evolutionary Computation}, 6:182--197,
  2002.

\bibitem{ADefazio_FBach_SLacoste-Julien_2014}
A.~Defazio, F.~Bach, and S.~Lacoste-Julien.
\newblock {SAGA: A fast incremental gradient method with support for
  non-strongly convex composite objectives}.
\newblock In {\em Advances in Neural Information Processing Systems}, pages
  1646--1654, 2014.

\bibitem{JADesideri_2012}
J.~A. D\'esid\'eri.
\newblock Multiple-gradient descent algorithm ({MGDA}) for multiobjective
  optimization.
\newblock {\em C. R. Math. Acad. Sci. Paris}, 350:313--318, 2012.

\bibitem{JADesideri_2014}
J.~A. D{\'e}sid{\'e}ri.
\newblock Multiple-gradient descent algorithm for {P}areto-front
  identification.
\newblock In {\em Modeling, Simulation and Optimization for Science and
  Technology}, pages 41--58. Springer, Dordrecht, 2014.

\bibitem{LGDrummond_ANIusem_2004}
L.~G. Drummond and A.~N. Iusem.
\newblock A projected gradient method for vector optimization problems.
\newblock {\em Comput. Optim. Appl.}, 28:5--29, 2004.

\bibitem{LGDrummond_FMPRaupp_BFSvaiter_2014}
L.~G. Drummond, F.~M.~P. Raupp, and B.~F. Svaiter.
\newblock A quadratically convergent {N}ewton method for vector optimization.
\newblock {\em Optimization}, 63:661--677, 2014.

\bibitem{LGDrummond_BFSvaiter_2005}
L.~G. Drummond and B.~F. Svaiter.
\newblock A steepest descent method for vector optimization.
\newblock {\em J. Comput. Appl. Math.}, 175:395--414, 2005.

\bibitem{MEhrgott_2005}
M.~Ehrgott.
\newblock {\em Multicriteria Optimization}, volume 491.
\newblock Springer Science \& Business Media, Berlin, 2005.

\bibitem{JFliege_LGDrummond_BFSvaiter_2009}
J.~Fliege, L.~G. Drummond, and B.~F. Svaiter.
\newblock Newton's method for multiobjective optimization.
\newblock {\em SIAM J. Optim.}, 20:602--626, 2009.

\bibitem{JFliege_BFSvaiter_2000}
J.~Fliege and B.~F. Svaiter.
\newblock Steepest descent methods for multicriteria optimization.
\newblock {\em Math. Methods Oper. Res.}, 51:479--494, 2000.

\bibitem{JFliege_AIFVaz_LNVicente_2018}
J.~Fliege, A.~I.~F. Vaz, and L.~N. Vicente.
\newblock Complexity of gradient descent for multiobjective optimization.
\newblock {\em to appear in Optim. Methods Softw.}, 2018.

\bibitem{JEFreund_1962}
J.~E. Freund.
\newblock {\em Mathematical Statistics}.
\newblock Prentice-Hall, Englewood Cliffs, N.J., 1962.

\bibitem{EHFukuda_LMGDrummond_2014}
E.~H. Fukuda and L.~M.~G. Drummond.
\newblock A survey on multiobjective descent methods.
\newblock {\em Pesquisa Operacional}, 34:585--620, 2014.

\bibitem{SGass_TSaaty_1955}
S.~Gass and T.~Saaty.
\newblock The computational algorithm for the parametric objective function.
\newblock {\em Nav. Res. Logist. Q.}, 2:39--45, 1955.

\bibitem{MGendreau_OJabali_WRei_2014}
M.~Gendreau, O.~Jabali, and W.~Rei.
\newblock Chapter 8: Stochastic vehicle routing problems.
\newblock In {\em Vehicle Routing: Problems, Methods, and Applications, Second
  Edition}, pages 213--239. SIAM, 2014.

\bibitem{AMGeoffrion_1968}
A.~M. Geoffrion.
\newblock Proper efficiency and the theory of vector maximization.
\newblock {\em J. Math. Anal. Appl.}, 22:618--630, 1968.

\bibitem{WJGutjahr_APichler_2016}
W.~J. Gutjahr and A.~Pichler.
\newblock Stochastic multi-objective optimization: a survey on non-scalarizing
  methods.
\newblock {\em Annals of Operations Research}, 236:475--499, 2016.

\bibitem{YVHaimes_1971}
Y.~V. Haimes.
\newblock On a bicriterion formulation of the problems of integrated system
  identification and system optimization.
\newblock {\em IEEE Transactions on Systems, Man, and Cybernetics}, 1:296--297,
  1971.

\bibitem{MHardt_EPrice_NSrebro_2016}
M.~Hardt, E.~Price, and N.~Srebro.
\newblock Equality of opportunity in supervised learning.
\newblock In {\em Advances in neural information processing systems}, pages
  3315--3323, 2016.

\bibitem{RJohnson_TZhang_2013}
R.~Johnson and T.~Zhang.
\newblock Accelerating stochastic gradient descent using predictive variance
  reduction.
\newblock In {\em NIPS}, pages 315--323, 2013.

\bibitem{AJKleywegt_etal_2002}
A.~J. Kleywegt, A.~Shapiro, and T.~Homem de~Mello.
\newblock The sample average approximation method for stochastic discrete
  optimization.
\newblock {\em SIAM J. Optim.}, 12:479--502, 2002.

\bibitem{SLiu_LNVicente_2020}
S.~Liu and L.~N. Vicente.
\newblock Accuracy and fairness trade-offs in machine learning: A stochastic
  multi-objective approach.
\newblock {\em ISE Technical Report 20T-016, Lehigh University}, 2020.

\bibitem{KMiettinen_2012}
K.~Miettinen.
\newblock {\em Nonlinear Multiobjective Optimization}, volume~12.
\newblock Springer Science \& Business Media, New York, 2012.

\bibitem{ANemirovski_2009}
A.~Nemirovski, A.~Juditsky, G.~Lan, and A.~Shapiro.
\newblock Robust stochastic approximation approach to stochastic programming.
\newblock {\em SIAM J. Optim.}, 19:1574--1609, 2009.

\bibitem{JOyola_HArntzen_DLWoodruff_2018}
J.~Oyola, H.~Arntzen, and D.~L. Woodruff.
\newblock The stochastic vehicle routing problem, a literature review, part i:
  models.
\newblock {\em EURO Journal on Transportation and Logistics}, 7:193--221, 2018.

\bibitem{LRLucambioPerez_LFPrudente_2018}
L.~R.~Lucambio P\'erez and L.~F. Prudente.
\newblock Nonlinear conjugate gradient methods for vector optimization.
\newblock {\em SIAM J. Optim.}, 28:2690--2720, 2018.

\bibitem{BTPolyak_ABJuditsky_1992}
B.~T. Polyak and A.~B. Juditsky.
\newblock Acceleration of stochastic approximation by averaging.
\newblock {\em SIAM J. Control Optim.}, 30:838--855, 1992.

\bibitem{SQu_MGoh_BLiang_2013}
S.~Qu, M.~Goh, and B.~Liang.
\newblock Trust region methods for solving multiobjective optimisation.
\newblock {\em Optim. Methods Softw.}, 28:796--811, 2013.

\bibitem{MQuentin_PFabrice_JADesideri_2018}
M.~Quentin, P.~Fabrice, and J.~A. D\'esid\'eri.
\newblock A stochastic multiple gradient descent algorithm.
\newblock {\em European J. Oper. Res.}, 271:808 -- 817, 2018.

\bibitem{HRobbins_SMonro_1951}
H.~Robbins and S.~Monro.
\newblock A stochastic approximation method.
\newblock {\em Ann. Math. Statist.}, 22:400--407, 1951.

\bibitem{JSacks_1958}
J.~Sacks.
\newblock Asymptotic distribution of stochastic approximation procedures.
\newblock {\em Ann. Math. Statist.}, 29:373--405, 1958.

\bibitem{SShalev-Shwartz_etal_2011}
S.~Shalev-Shwartz, Y.~Singer, N.~Srebro, and A.~Cotter.
\newblock Pegasos: Primal estimated sub-gradient solver for svm.
\newblock {\em Math. Program.}, 127:3--30, 2011.

\bibitem{AShapiro_2003}
A.~Shapiro.
\newblock Monte {C}arlo sampling methods.
\newblock {\em Handbooks in Operations Research and Management Science},
  10:353--425, 2003.

\bibitem{SVerma_JRubin_2018}
S.~Verma and J.~Rubin.
\newblock Fairness definitions explained.
\newblock In {\em 2018 IEEE/ACM International Workshop on Software Fairness
  (FairWare)}, pages 1--7. IEEE, 2018.

\bibitem{KDVillacorta_PROliveira_ASoubeyran_2014}
K.~D. Villacorta, P.~R. Oliveira, and A.~Soubeyran.
\newblock A trust-region method for unconstrained multiobjective problems with
  applications in satisficing processes.
\newblock {\em J. Optim. Theory Appl.}, 160:865--889, 2014.

\bibitem{BWoodworth_etal_2017}
B.~Woodworth, S.~Gunasekar, M.~I. Ohannessian, and N.~Srebro.
\newblock Learning non-discriminatory predictors.
\newblock In {\em Conference on Learning Theory}, pages 1920--1953, 2017.

\bibitem{MBZafar_etal_2017b}
M.~B. Zafar, I.~Valera, M.~Gomez Rodriguez, and K.~P. Gummadi.
\newblock Fairness constraints: Mechanisms for fair classification.
\newblock In {\em Artificial Intelligence and Statistics}, pages 962--970,
  2017.

\bibitem{RZemel_etal_2013}
R.~Zemel, Y.~Wu, K.~Swersky, T.~Pitassi, and C.~Dwork.
\newblock Learning fair representations.
\newblock In {\em International Conference on Machine Learning}, pages
  325--333, 2013.

\end{thebibliography}

\end{document}